\documentclass[article,onefignum,onetabnum]{siamart190516}

\usepackage{amssymb}

\newsiamremark{remark}{Remark}
\newsiamremark{assumption}{Assumption}
\crefname{assumption}{Assumption}{Assumptions}

\headers{DPP and HJB Equations for Fractional-Order Systems}{M. I. Gomoyunov}

\title{Dynamic Programming Principle and Hamilton--Jacobi--Bellman Equations for Fractional-Order Systems
\thanks{Submitted to the editors 05.08.2019.
\funding{This work was supported by RSF (project no.~19-71-00073).}}}

\author{Mikhail I. Gomoyunov
    \thanks{Krasovskii Institute of Mathematics and Mechanics of the Ural Branch of the Russian Academy of Sciences, Ekaterinburg, Russia;
    Ural Federal University, Ekaterinburg, Russia
    (\email{m.i.gomoyunov@gmail.com}).}}

\def\Val{\rho} 
\def\diam{\operatorname{diam}} 
\def\dist{\operatorname{dist}} 
\def\H{\mathcal{H}} 
\def\mU{\mathcal{U}} 
\def\rd{\mathrm{d}} 
\def\ess{\operatorname{ess}} 
\def\Linf{L^\infty}
\def\AC{{AC^\alpha}}
\def\C{{C}}

\begin{document}

\maketitle

\begin{abstract}
    We consider a Bolza-type optimal control problem for a dynamical system described by a fractional differential equation with the Caputo derivative of an order $\alpha \in (0, 1)$.
    The value of this problem is introduced as a functional in a suitable space of histories of motions.
    We prove that this functional satisfies the dynamic programming principle.
    Based on a new notion of coinvariant derivatives of the order $\alpha$, we associate the considered optimal control problem with a Hamilton--Jacobi--Bellman equation.
    Under certain smoothness assumptions, we establish a connection between the value functional and a solution to this equation.
    Moreover, we propose a way of constructing optimal feedback controls.
    The paper concludes with an example.
\end{abstract}



\section{Introduction}

The dynamic programming principle, along with the Pontryagin maximum principle, is one of the fundamental and most effective tools for studying and constructing solutions of various optimal control problems.
The goal of the paper is to extend the dynamic programming principle to the case of fractional-order dynamical systems.
More precisely, we focus on a Bolza-type optimal control problem for a system described by a fractional differential equation with the Caputo derivative of an order $\alpha \in (0, 1)$.
Such problems arise in a wide range of applications, including, e.g., biology \cite{Toledo-Hernandez_et_all_2014}, chemistry \cite{Flores-Tlacuahuac_Biegler_2014}, economics \cite{Hajipour_Hajipour_Baleanu_2018}, electrical engineering \cite{Kaczorek_2016}, and medicine \cite{Hossein_Mohsen_2018}.
For the basics of fractional calculus and the theory of fractional-order differential equations, the reader is referred to, e.g., \cite{Samko_Kilbas_Marichev_1993,Kilbas_Srivastava_Trujillo_2006,Diethelm_2010}.

In order to formulate the dynamic programming principle, for every intermediate time $t \in (0, T)$, it is necessary to introduce an auxiliary optimal control problem (sub-problem) with this time $t$ considered as the initial one.
This auxiliary problem should be consistent with the original one in the sense that the optimality principle (see, e.g., \cite[\S~3]{Bellman_1957}) should be satisfied.
The analysis of simple examples shows (see \cref{Section_Example_1} for a detailed discussion) that this sub-problem should be formulated for the original dynamical system but with the whole history of the motion $x_t(\tau) = x(\tau)$, $\tau \in [0, t]$, taken as the initial data at the time $t$, and, hence, the pair $(t, x_t(\cdot))$ should be regarded as a position of the system formed by the time $t$.
This is explained by the fact that the Caputo derivative has a nonlocal character (see, e.g., \cite{Tarasov_2018}), i.e., its value $(^C D^\alpha x)(t)$ at the time $t$ depends essentially on all previous values $x(\tau)$, $\tau \in [0, t]$.
Thus, we conclude that, in the optimal control problem under consideration, the value (the optimal result) should be introduced as a functional $\Val(t, w(\cdot))$, where the function $w(\tau),$ $\tau \in [0, t]$, is treated as a history of a motion of the system on $[0, t]$.
This circumstance significantly differs the approach developed in the paper from the existing works on the dynamic programming principle for fractional-order systems (see, e.g., \cite{Jumarie_2007,Rakhshan_Effati_Kamyad_2018,Razminia_Mehdi_Torres_2018}), in which the value is a function $\Val(t, x)$, where $x$ is treated as a value of the state vector at the time $t$.
Let us note that this approach, involving the dependence of the value on a history of a motion, is generally accepted in the control theory for functional-differential systems, and, therefore, due to the relationship between fractional-order systems and functional-differential systems of a neutral type (see, e.g., \cite[Sect.~4]{Gomoyunov_2019_DGAA}), it seems natural to use this approach in the considered problem, too.

In many control problems, for finding the value and constructing optimal feedback controls, it is often more convenient to apply not the dynamic programming principle itself, but its infinitesimal form expressed in terms of a Hamilton--Jacobi--Bellman equation.
The second goal of the paper is to derive and study a Hamilton--Jacobi--Bellman equation that corresponds to the considered optimal control problem.

In order to do this, we introduce a new notion of coinvariant ($ci$-) differentiation of the order $\alpha$ of functionals $\varphi(t, w(\cdot))$, where $(t, w(\cdot))$ is a position of the system from a suitable functional space.
This notion naturally agrees (when, formally, $\alpha = 1$) with the notion of $ci$-differentiation, which has proved to be a convenient tool for developing the theory of Hamilton--Jacobi equations for various control problems in functional-differential systems of retarded (see, e.g., \cite{Kim_1999,Lukoyanov_2000_JAMM,Lukoyanov_2003,Lukoyanov_2011_Eng,Kaise_2015,Plaksin_2019_IFAC} and the references therein) and neutral (see, e.g., \cite{Lukoyanov_Gomoyunov_Plaksin_2017_Doklady,Gomoyunov_Plaksin_2018_IFAC}) types (see also \cite{Kaise_Kato_Takahashi_2018}).
Let us note that, e.g., in \cite{Aubin_Haddad_2002,Dupire_2009,Pham_Zhang_2014,Bayraktar_Keller_2018,Saporito_2019} (see also the references therein), some other notions of differentiation of functionals are used, which, in some sense, are close to $ci$-differentiation.
The main advantage of the introduced notion of fractional $ci$-differentiation is the fact that it allows us to obtain a simple formula for the total derivative of the value functional along a motion of the fractional-order system.

Based on this formula and the dynamic programming principle, we associate the considered optimal control problem with the Hamilton--Jacobi--Bellman equation with the $ci$-derivatives of the order $\alpha$ and establish a connection between them.
Namely, we show that, if the value functional is sufficiently smooth ($ci$-smooth of the order $\alpha$), then it satisfies this equation.
On the other hand, we prove that if the Cauchy problem for this equation and a natural right-end condition admits a $ci$-smooth of the order $\alpha$ solution, then it coincides with the value functional, and, moreover, we can construct an optimal feedback control strategy by using the extremal shift in the direction of the $ci$-gradients of the order $\alpha$ of this solution.
Here, we use a quite general notion of feedback control strategies that goes back to the positional approach in differential games \cite{Krasovskii_Subbotin_1988,Krasovskii_Krasovskii_1995} (see also \cite{Osipov_1971,Lukoyanov_2000_JAMM,Lukoyanov_2003,Lukoyanov_2011_Eng} and \cite{Gomoyunov_2019_Trudy_Eng,Gomoyunov_2019_DGAA}).
Finally, we illustrate the obtained results by an example.

The rest of the paper is organized as follows.
In \cref{Section_Preliminaries}, we introduce the notations, recall the definitions of fractional-order integrals and derivatives, and give some of their properties.
In \cref{Section_Optimal_control_problem_initial}, we describe the optimal control problem under consideration.
\Cref{Section_Example_1} is devoted to a discussion of the dynamic programming principle for fractional-order systems on a simple example.
Further, in \cref{Section_Optimal_control_problem}, a general statement of the optimal control problem with an arbitrary initial position $(t, w(\cdot))$ is given, and the value functional is defined.
The dynamic programming principle is proved in \cref{section_dynamic_programming_principle}.
In \Cref{section_properties_of_motions,section_continuity}, we establish some technical properties of motions of the system and show that the value functional is continuous.
In \cref{section_derivatives}, we introduce the notion of $ci$-differentiability of the order $\alpha$.
\Cref{section_HJB,section_optimal_control_strategy} deal with the Hamilton--Jacobi--Bellman equation and its connection with the considered optimal control problem.
Based on the obtained results, in \cref{Section_Example_2}, we solve the example from \cref{Section_Example_1}.
Concluding remarks are given in \cref{section_conclusion}.

\section{Notations and definitions}
\label{Section_Preliminaries}

Let $T > 0$, $n \in \mathbb{N}$, and $\alpha \in (0, 1)$ be fixed throughout the paper.
By $\|\cdot\|$ and $\langle \cdot, \cdot \rangle$, we denote the Euclidian norm and the inner product in $\mathbb{R}^n$.

For every $t \in [0, T]$, let $\Linf([0, t], \mathbb{R}^n)$ be the set of (Lebesgue) measurable functions $\psi: [0, t] \rightarrow \mathbb{R}^n$ such that $\|\psi(\cdot)\|_{[0, t]} < \infty$, where $\|\psi(\cdot)\|_{[0, t]} = \ess \sup_{\tau \in [0, t]} \|\psi(\tau)\|$ if $t > 0$ and $\|\psi(\cdot)\|_{[0, t]} = \|\psi(0)\|$ if $t = 0$.
For a function $\psi(\cdot) \in \Linf([0, t], \mathbb{R}^n)$, the (left-sided) Riemann--Liouville fractional integral of the order $\alpha$ is defined by
\begin{displaymath}
    (I^\alpha \psi)(\tau)
    = \frac{1}{\Gamma(\alpha)} \int_{0}^{\tau} \frac{\psi(\xi)}{(\tau - \xi)^{1 - \alpha}} \, \rd \xi,
    \quad \tau \in [0, t],
\end{displaymath}
where $\Gamma$ is the gamma function.
According to, e.g., \cite[Theorem~3.6 and Remark~3.3]{Samko_Kilbas_Marichev_1993} (see also \cite[Theorem~2.6]{Diethelm_2010}), for $H_\alpha = 2 / \Gamma(\alpha + 1)$, we have
\begin{equation} \label{H_alpha}
    \| (I^\alpha \psi)(\tau) - (I^\alpha \psi)(\tau^\prime) \|
    \leq H_\alpha \|\psi(\cdot)\|_{[0, t]} |\tau - \tau^\prime|^\alpha,
    \quad \tau, \tau^\prime \in [0, t].
\end{equation}

Let us denote by $\AC([0, t], \mathbb{R}^n)$ the set of functions $x: [0, t] \rightarrow \mathbb{R}^n$ that can be represented in the form
\begin{equation} \label{AC^alpha}
    x(\tau)
    = x(0) + (I^\alpha \psi)(\tau),
    \quad \tau \in [0, t],
\end{equation}
for some function $\psi(\cdot) \in \Linf([0, t], \mathbb{R}^n)$.
This set $\AC([0, t], \mathbb{R}^n)$ is considered as a subset of the space $\C([0, t], \mathbb{R}^n)$ of continuous functions endowed with the uniform norm $\|\cdot\|_{[0, t]}$.
In the case $t = 0$, the set $\AC([0, t], \mathbb{R}^n)$ can be identified with $\mathbb{R}^n$.

Let $x(\cdot) \in \AC([0, t], \mathbb{R}^n),$ and let $\psi(\cdot) \in \Linf([0, t], \mathbb{R}^n)$ be such that \cref{AC^alpha} is valid.
Then, according to, e.g., \cite[Theorem~2.5]{Samko_Kilbas_Marichev_1993} (see also \cite[Theorem~2.2]{Diethelm_2010}), we obtain
\begin{equation} \label{int_psi}
    \big(I^{1 - \alpha} (x(\cdot) - x(0))\big)(\tau)
    = \big(I^{1 - \alpha} (I^\alpha \psi)\big)(\tau)
    = \int_{0}^{\tau} \psi(\xi) \, \rd \xi,
    \quad \tau \in [0, t].
\end{equation}
Hence, the (left-sided) Caputo fractional derivative of $x(\cdot)$ of the order $\alpha$ defined by
\begin{equation} \label{Caputo_derivative}
    (^C D^\alpha x) (\tau)
    = \frac{\rd}{\rd \tau} \big(I^{1 - \alpha} (x(\cdot) - x(0))\big)(\tau)
    = \frac{1}{\Gamma(1 - \alpha)} \frac{\rd}{\rd \tau} \int_{0}^{\tau} \frac{x(\xi) - x(0)}{(\tau - \xi)^{\alpha}} \, \rd \xi
\end{equation}
exists for almost every (a.e.) $\tau \in [0, t]$, and, moreover, $(^C D^\alpha x)(\tau) = \psi(\tau)$ for a.e. $\tau \in [0, t]$.
In particular, we have $\|(^C D^\alpha x)(\cdot)\|_{[0, t]} = \|\psi(\cdot)\|_{[0, t]} < \infty$, and, from \cref{AC^alpha,int_psi}, it follows that the equalities
\begin{align}
    & x(\tau)
    = x(0) + \big(I^\alpha (^C D^\alpha x) \big)(\tau), \label{ID} \\
    & \big(I^{1 - \alpha} (x(\cdot) - x(0))\big)(\tau)
    = \int_{0}^{\tau} (^C D^\alpha x) (\xi) \, \rd \xi \label{int_D^alpha}
\end{align}
hold for every $\tau \in [0, t]$.

\section{Optimal control problem}
\label{Section_Optimal_control_problem_initial}

Let us consider a dynamical system which motion is described by the fractional differential equation
\begin{equation} \label{system}
    (^C D^\alpha x)(\tau)
    = f(\tau, x(\tau), u(\tau)),
    \quad x(\tau) \in \mathbb{R}^n, \quad u(\tau) \in P, \quad \tau \in [0, T],
\end{equation}
with the initial condition
\begin{equation} \label{initial_condition}
    x(0)
    = w_0.
\end{equation}
Here, $\tau$ is time;
$x(\tau)$ and $u(\tau)$ are the current values of the state and control vectors, respectively;
$(^C D^\alpha x)(\tau)$ is the Caputo derivative of the order $\alpha$ (see \cref{Caputo_derivative});
$P \subset \mathbb{R}^{n_u}$ is a compact set, $n_u \in \mathbb{N}$;
$w_0 \in \mathbb{R}^n$ determines the initial value of the state vector.

\begin{assumption} \label{Assumption_f}
    The function $f: [0, T] \times \mathbb{R}^n \times P \rightarrow \mathbb{R}^n$ in \cref{system} satisfies the following conditions:
    (a) $f$ is continuous;
    (b) $f$ is locally Lipschitz continuous in the second argument, i.e., for any $R > 0,$ there exists $\lambda_f > 0$ such that
    \begin{displaymath}
        \|f(\tau, x, u) - f(\tau, x^\prime, u)\| \leq \lambda_f \|x - x^\prime\|,
        \quad \tau \in [0, T], \quad x, x^\prime \in B(R), \quad u \in P,
    \end{displaymath}
    where we denote $B(R) = \{x \in \mathbb{R}^n: \, \|x\| \leq R\}$;
    (c) $f$ has sublinear growth in the second argument, i.e., there exists $c_f > 0$ such that
    \begin{equation} \label{c_f}
        \|f(\tau, x, u)\| \leq (1 + \|x\|) c_f,
        \quad \tau \in [0, T], \quad x \in \mathbb{R}^n, \quad u \in P.
    \end{equation}
\end{assumption}

By an admissible control $u(\cdot)$, we mean a measurable function $u: [0, T] \rightarrow P$.
The set of all such controls is denoted by $\mU$.
A motion of system \cref{system} with initial condition \cref{initial_condition} that corresponds to a control $u(\cdot) \in \mU$ is a function $x(\cdot) \in \AC([0, T], \mathbb{R}^n)$ that satisfies the equality in \cref{initial_condition} and, together with $u(\cdot)$, the differential equation in \cref{system} for a.e. $\tau \in [0, T]$.
Due to \cref{Assumption_f}, such a motion $x(\cdot)$ exists and is unique (see, e.g., \cite[Proposition~5.1]{Gomoyunov_2019_FCAA}).
Moreover, $x(\cdot)$ is a unique function from $\C([0, T], \mathbb{R}^n)$ that satisfies the integral equation
\begin{equation} \label{integral_equation_0}
    x(\tau) = w_0 + \frac{1}{\Gamma(\alpha)} \int_{0}^{\tau} \frac{f(\xi, x(\xi), u(\xi))}{(\tau - \xi)^{1 - \alpha}} \, \rd \xi,
    \quad \tau \in [0, T].
\end{equation}

The goal of control is to minimize the cost functional
\begin{equation} \label{cost_functional}
    J(0, w_0, u(\cdot))
    = \sigma(x(T)) + \int_{0}^{T} \chi(\tau, x(\tau), u(\tau)) \, \rd \tau,
\end{equation}
where $x(\cdot)$ is the motion of system \cref{system} with initial condition \cref{initial_condition} that corresponds to a control $u(\cdot) \in \mU$.
The first argument in the notation $J(0, w_0, u(\cdot))$ is added to indicate the initial time, which is convenient for further constructions (see \cref{Section_Optimal_control_problem}).

\begin{assumption} \label{Assumption_sigma_chi}
    The functions $\sigma: \mathbb{R}^n \rightarrow \mathbb{R}$ and $\chi: [0, T] \times \mathbb{R}^n \times P \rightarrow \mathbb{R}$ in \cref{cost_functional} are continuous.
\end{assumption}

The value of the optimal control problem \cref{system,initial_condition,cost_functional} is defined by
\begin{displaymath}
    \Val(0, w_0)
    = \inf_{u(\cdot) \in \mU} J(0, w_0, u(\cdot)).
\end{displaymath}
A control $u^\circ(\cdot) \in \mU$ is called optimal if $J(0, w_0, u^\circ(\cdot)) = \Val(0, w_0)$.
However, since the problem is considered under rather general assumptions, one can not expect that an optimal control exists.
Therefore, we are interested in finding, for every sufficiently small $\varepsilon > 0$, an $\varepsilon$-optimal control $u^{(\varepsilon)}(\cdot) \in \mU$, which satisfies the inequality
\begin{displaymath}
    J(0, w_0, u^{(\varepsilon)}(\cdot))
    \leq \Val(0, w_0) + \varepsilon.
\end{displaymath}

The goal of the paper is to extend the dynamic programming principle to the optimal control problem \cref{system,initial_condition,cost_functional} and, furthermore, derive the corresponding Hamilton--Jacobi--Bellman equation.

The dynamic programming principle is based on the principle of optimality (see, e.g., \cite[\S~3]{Bellman_1957}), which, in particular, states the following.
An optimal control $u^\circ(\cdot)$ has the property that, for any intermediate time $t \in (0, T)$, the control $u^\ast(\tau) = u^\circ(\tau)$, $\tau \in [t, T]$, must be optimal for the sub-problem on the remaining time interval $[t, T]$ with the initial position resulting from the previous control $u_\ast(\tau) = u^\circ(\tau)$, $\tau \in [0, t]$.
Thus, in order to apply this principle to the problem \cref{system,initial_condition,cost_functional}, it is necessary to define what is meant by ``sub-problem'' and by ``position''.
One can propose several quite reasonable ways to answer this question.
In the next section, we consider a simple example to illustrate three of them.
Let us note that a complete solution to this example is given in \cref{Section_Example_2} on the basis of the results of the paper.

\section{Example}
\label{Section_Example_1}

Let the optimal control problem be described by the system
\begin{equation}\label{system_example}
    (^C D^\alpha x)(\tau)
    = \Gamma(\alpha + 1) u(\tau),
    \quad x(\tau) \in \mathbb{R}, \quad |u(\tau)| \leq 1, \quad \tau \in [0, 2],
\end{equation}
with the initial condition
\begin{equation}\label{initial_condition_example}
    x(0)
    = w_0
    = 2^{\alpha - 1} + 1
\end{equation}
and the cost functional
\begin{equation}\label{cost_functional_example}
    J(0, w_0, u(\cdot))
    = x^2(2).
\end{equation}
In \cref{system_example}, the coefficient $\Gamma(\alpha + 1)$ is added to simplify the formulas below.

By direct calculations, one can show that $u^\circ(\tau) = - 1$, $\tau \in [0, 2]$, is the optimal control, and the corresponding motion of \cref{system_example,initial_condition_example} is $x^\circ(\tau) = 2^{\alpha - 1} + 1 - \tau^\alpha > 0$, $\tau \in [0, 2]$.
In particular, we have $\Val(0, w_0) = (1 - 2^{\alpha - 1})^2 > 0$.

Taking into account the discussion of the optimality principle in \cref{Section_Optimal_control_problem_initial}, let us choose $t = 1$ as an intermediate time and propose the following three approaches to defining the appropriate notions of ``sub-problem'' and ``position''.

\subsection{First approach}
\label{Subsection_First_approach}

    As a sub-problem on the time interval $[1, 2]$, let us consider the optimal control problem for the system
    \begin{equation} \label{system_example_2}
        (^C D_1^\alpha y)(\tau)
        = \Gamma(\alpha + 1) u(\tau),
        \quad y(\tau) \in \mathbb{R}, \quad |u(\tau)| \leq 1, \quad \tau \in [1, 2],
    \end{equation}
    with the initial condition
    \begin{equation} \label{initial_condition_example_2}
        y(1)
        = x^\circ(1)
        = 2^{\alpha - 1}
    \end{equation}
    and the cost functional
    \begin{equation} \label{cost_functional_example_2}
        J(1, x^\circ(1), u(\cdot))
        = y^2(2).
    \end{equation}
    Here, by $(^C D_1^\alpha y)(\tau)$, we denote the Caputo derivative (see \cref{Caputo_derivative}) with the lower terminal $1$, which is defined by
    \begin{displaymath}
        (^C D_1^\alpha y) (\tau)
        = \frac{1}{\Gamma(1 - \alpha)} \frac{\rd}{\rd \tau}
        \int_{1}^{\tau} \frac{y(\xi) - y(1)}{(\tau - \xi)^{\alpha}} \, \rd \xi,
        \quad \tau \in [1, 2].
    \end{displaymath}
    Thus, this sub-problem is the same as the original one \cref{system_example,initial_condition_example,cost_functional_example} but with the initial time formally changed to $1$, and the role of the position resulting from the control $u_\ast(\tau) = u^\circ(\tau) = - 1$, $\tau \in [0, 1]$, is played by the pair $(1, x^\circ(1))$.

    Let us show that the optimality principle is not satisfied within this approach, i.e., the control $u^\ast(\tau) = u^\circ(\tau) = - 1$, $\tau \in [1, 2]$, is not optimal in the problem \cref{system_example_2,initial_condition_example_2,cost_functional_example_2}.
    For the motion $y^\ast(\cdot)$ of \cref{system_example_2,initial_condition_example_2} that corresponds to $u^\ast(\cdot)$, we have
    \begin{equation} \label{y^ast}
        y^\ast(\tau)
        = x^\circ(1) + \frac{1}{\Gamma(\alpha)} \int_{1}^{\tau} \frac{\Gamma(\alpha + 1) u^\ast(\xi)}{(\tau - \xi)^{1 - \alpha}} \, \rd \xi
        = 2^{\alpha - 1} - (\tau - 1)^\alpha,
        \quad \tau \in [1, 2],
    \end{equation}
    and, therefore, $J(1, x^\circ(1), u^\ast(\cdot)) = (2^{\alpha - 1} - 1)^2 > 0$.
    However, choosing the control
    \begin{equation} \label{overline_u}
        \bar{u}(\tau) =
        \begin{cases}
            - 1, & \mbox{if } \tau \in [1, \theta), \\
            0, & \mbox{if } \tau \in [\theta, 2],
        \end{cases}
        \quad \theta = 2 - (1 - 2^{\alpha - 1})^{1 / \alpha},
    \end{equation}
    for the corresponding motion $\bar{y}(\cdot)$ of \cref{system_example_2,initial_condition_example_2}, we obtain $\bar{y}(2) = 0$, and, hence, $J(1, x^\circ(1), \bar{u}(\cdot)) = 0 < J(1, x^\circ(1), u^\ast(\cdot))$.
    Thus, the control $u^\ast(\cdot)$ is not optimal.

\subsection{Second approach}

    It follows from the previous arguments that we can not simply change the original system by substituting the fractional derivative with another lower terminal.
    Taking this into account, in a sub-problem, let us consider the same system \cref{system_example}.
    According to \cref{Caputo_derivative}, the value of the fractional derivative $(^C D^\alpha x)(\tau)$ for $\tau \in [0, 2]$ depends on the values $x(\xi)$ for all $\xi \in [0, \tau]$.
    Hence, at least formally, in order to correctly formulate an initial value problem for the differential equation in \cref{system_example} with the initial time $t = 1$, one should specify the values $x(\xi)$ for all $\xi \in [0, 1]$.
    However, all these values are known since they have already been realized during the time interval $[0, 1]$.
    Then, if we denote the history of the motion $x^\circ(\cdot)$ on $[0, 1]$ by $w^\circ(\tau) = x^\circ(\tau)$, $\tau \in [0, 1]$, we obtain the following initial condition:
    \begin{equation} \label{initial_condition_example_3}
        x(\tau)
        = w^\circ(\tau),
        \quad \tau \in [0, 1].
    \end{equation}
    Thus, we come to the optimal control problem (sub-problem) for system \cref{system_example} with initial condition \cref{initial_condition_example_3} and the cost functional
    \begin{equation} \label{cost_functional_example_3}
        J(1, w^\circ(\cdot), u(\cdot))
        = x^2(2).
    \end{equation}
    Here, in accordance with \cref{integral_equation_0}, the motion $x(\cdot)$ for $\tau \in [1, 2]$ is given by
    \begin{equation} \label{motion_example_3}
        x(\tau)
        = w^\circ(0) + \frac{1}{\Gamma(\alpha)} \int_{0}^{1} \frac{(^C D^\alpha w^\circ)(\xi)}{(\tau - \xi)^{1 - \alpha}} \, \rd \xi
        + \frac{1}{\Gamma(\alpha)} \int_{1}^{\tau} \frac{\Gamma(\alpha + 1) u(\xi)}{(\tau - \xi)^{1 - \alpha}} \, \rd \xi.
    \end{equation}

    Nevertheless, one might assume that a solution to the described sub-problem depends not on the whole history $w^\circ(\tau)$, $\tau \in [0, 1]$, of the motion $x^\circ(\cdot)$, but only on the single value $w^\circ(1) = x^\circ(1)$, and, therefore, the pair $(1, x^\circ(1))$ still plays the role of the position (as in \cref{Subsection_First_approach}).
    In general, this assumption may be motivated by a specific form of system \cref{system_example} and/or more subtle results on the uniqueness of solutions of fractional differential equations (see, e.g., \cite{Cong_Tuan_2017}).
    However, let us show that the optimality principle is not satisfied within this approach too, i.e., there exists another history $\tilde{w}(\cdot)$ such that $\tilde{w}(1) = x^\circ(1)$ but the control $u^\ast(\tau) = u^\circ(\tau) = - 1$, $\tau \in [1, 2]$, is not optimal in the problem for system \cref{system_example} with the initial condition
    \begin{equation} \label{initial_condition_example_4}
        x(\tau)
        = \tilde{w}(\tau),
        \quad \tau \in [0, 1],
    \end{equation}
    and the cost functional
    \begin{displaymath}
        J(1, \tilde{w}(\cdot), u(\cdot))
        = x^2(2),
    \end{displaymath}
    where $x(\cdot)$ is defined by analogy with \cref{motion_example_3}.
    Namely, let $\tilde{w}(\tau) = x^\circ(1) = 2^{\alpha - 1}$, $\tau \in [0, 1]$.
    Let us note that this function can be considered as the history on $[0, 1]$ of the motion $\tilde{x}(\cdot)$ of system \cref{system_example} that corresponds to the initial condition $\tilde{x}(0) = 2^{\alpha - 1}$ and the control $\tilde{u}(\tau) = 0$, $\tau \in [0, 2]$.
    As in \cref{y^ast}, for the motion $x^\ast(\cdot)$ of \cref{system_example,initial_condition_example_4} that corresponds to $u^\ast(\cdot)$, we have $x^\ast(\tau) = 2^{\alpha - 1} - (\tau - 1)^\alpha$, $\tau \in [1, 2]$, and, hence, $J(1, \tilde{w}(\cdot), u^\ast(\cdot)) = (2^{\alpha - 1} - 1)^2 > 0$.
    On the other hand, for the motion $\bar{x}(\cdot)$ of \cref{system_example,initial_condition_example_4} that corresponds to $\bar{u}(\cdot)$ from \cref{overline_u}, we obtain $\bar{x}(2) = 0$, and, therefore, $J(1, \tilde{w}(\cdot), \bar{u}(\cdot)) = 0 < J(1, \tilde{w}(\cdot), u^\ast(\cdot))$.

\subsection{Third approach}

    Summarizing the above, we conclude that, as the position, we should consider the pair $(t, w^\circ(\cdot))$ with the whole history $w^\circ(\cdot)$ of the motion $x^\circ(\cdot)$ on the time interval $[0, t]$.
    In fact, it follows from \cref{Theorem_dynamic_programming_principle} below that the optimality principle is satisfied with the sub-problem defined by \cref{system_example,initial_condition_example_3,cost_functional_example_3}.
    Let us note that this approach goes back to the control theory for functional-differential systems.
    For more details on the relationship between fractional-order systems and functional-differential systems of a neutral type, the reader is referred to, e.g., \cite[Sect.~4]{Gomoyunov_2019_DGAA}.

\section{Optimal control problem: statement for an arbitrary position}
\label{Section_Optimal_control_problem}

Thus, in accordance with \cref{Section_Example_1} (see also \cite{Gomoyunov_2019_Trudy_Eng,Gomoyunov_2019_DGAA}), by a position of system \cref{system}, we mean a pair $(t, w(\cdot))$ consisting of a time $t \in [0, T]$ and a function $w(\cdot) \in \AC([0, t], \mathbb{R}^n)$ (see the definition of $\AC([0, t], \mathbb{R}^n)$ in \cref{Section_Preliminaries}), which is treated as a history of a motion $x(\cdot)$ of system \cref{system} on the time interval $[0, t]$.
The set of all such positions $(t, w(\cdot))$ is denoted by $G$.

In this section, for an arbitrary position $(t, w(\cdot)) \in G$ considered as the initial one, we formulate the optimal control problem (sub-problem) that is consistent with the original one \cref{system,initial_condition,cost_functional}.
Besides, we introduce some auxiliary notations.

Let $(t, w(\cdot)) \in G$ be fixed, and let $\theta \in [t, T]$.
The set $\mU(t, \theta)$ of admissible controls on $[t, \theta]$ consists of all measurable functions $u: [t, \theta] \rightarrow P$.
Due to \cref{Assumption_f}, for the position $(t, w(\cdot))$ and a control $u(\cdot) \in \mU(t, \theta)$, there exists a unique motion of system \cref{system}, which is a function $x(\cdot) \in \AC([0, \theta], \mathbb{R}^n)$ that satisfies the equality
\begin{equation} \label{initial_condition_general}
    x(\tau)
    = w(\tau),
    \quad \tau \in [0, t],
\end{equation}
and, together with $u(\cdot)$, the differential equation in \cref{system} for a.e. $\tau \in [t, \theta]$ (see, e.g., \cite[Proposition~2]{Gomoyunov_2019_DGAA}).
This motion $x(\cdot)$ is a unique function from $\C([0, \theta], \mathbb{R}^n)$ that satisfies \cref{initial_condition_general} and, for every $\tau \in [t, \theta]$, the integral equation
\begin{equation} \label{integral_equation}
    x(\tau) = w(0) + \frac{1}{\Gamma(\alpha)} \int_{0}^{t} \frac{(^C D^\alpha w)(\xi)}{(\tau - \xi)^{1 - \alpha}} \, \rd \xi
    + \frac{1}{\Gamma(\alpha)} \int_{t}^{\tau} \frac{f(\xi, x(\xi), u(\xi))}{(\tau - \xi)^{1 - \alpha}} \, \rd \xi.
\end{equation}
Below, for the motion $x(\cdot)$, we also use the notation $x(\cdot \mid t, w(\cdot), \theta, u(\cdot))$.
Besides, for every $\tau \in [0, \theta]$, by $x_\tau(\cdot)$, we denote the history of the motion $x(\cdot)$ on $[0, \tau]$ defined by
\begin{equation} \label{x_t}
    x_\tau(\xi)
    = x(\xi),
    \quad \xi \in [0, \tau].
\end{equation}
Let us note that the inclusions $(\tau, x_\tau(\cdot)) \in G$, $\tau \in [0, \theta]$, are valid.
Moreover, let us emphasize that, under the considered statement, motions of system \cref{system} satisfy the so-called semigroup property (see, e.g., \cite[Sect.~3.2]{Gomoyunov_2019_DGAA}).

In the optimal control problem for system \cref{system} with the initial position $(t, w(\cdot))$, by choosing a control $u(\cdot) \in \mU(t, T)$, we want to minimize the cost functional
\begin{equation} \label{cost_functional_general}
    J(t, w(\cdot), u(\cdot))
    = \sigma(x(T)) + \int_{t}^{T} \chi(\tau, x(\tau), u(\tau)) \, \rd \tau,
\end{equation}
where $x(\cdot) = x(\cdot \mid t, w(\cdot), T, u(\cdot))$.
Thus, the value of this problem is given by
\begin{equation} \label{Val}
    \Val(t, w(\cdot))
    = \inf_{u(\cdot) \in \mU(t, T)} J(t, w(\cdot), u(\cdot)),
\end{equation}
and we are interested in finding, for every $\varepsilon > 0$, an $\varepsilon$-optimal control $u^{(\varepsilon)}(\cdot) \in \mU(t, T)$:
\begin{equation} \label{varepsilon_optimal}
    J(t, w(\cdot), u^{(\varepsilon)}(\cdot))
    \leq \Val(t, w(\cdot)) + \varepsilon.
\end{equation}

Let us note that, in the case $t = 0$ and $w(0) = w_0$, the statement of the considered optimal control problem for an arbitrary initial position $(t, w(\cdot)) \in G$ agrees with the original one given in \cref{Section_Optimal_control_problem_initial}.

\section{Dynamic programming principle}
\label{section_dynamic_programming_principle}

Since relation \cref{Val} defines the value $\Val(t, w(\cdot))$ for every $(t, w(\cdot)) \in G$, we can consider the value functional $\Val: G \rightarrow \mathbb{R}$.
We use the term ``functional'' to emphasize that $\Val(t, w(\cdot))$ depends on the infinite-dimensional argument $w(\cdot)$, which is consistent with the discussion in \cref{Section_Example_1}.

The value functional $\Val$ satisfies the terminal condition
\begin{equation} \label{Val_T}
    \Val(T, w(\cdot))
    = \sigma(w(T)),
    \quad w(\cdot) \in \AC([0, T], \mathbb{R}^n),
\end{equation}
and the following dynamic programming principle.
\begin{theorem} \label{Theorem_dynamic_programming_principle}
    For any $(t, w(\cdot)) \in G$ and $\theta \in [t, T]$, the equality below holds:
    \begin{equation} \label{dynamic_programming_principle}
        \Val(t, w(\cdot))
        = \inf_{u(\cdot) \in \mU(t, \theta)} \Big( \Val(\theta, x(\cdot)) + \int_{t}^{\theta} \chi(\tau, x(\tau), u(\tau)) \, \rd \tau \Big),
    \end{equation}
    where $x(\cdot) = x(\cdot \mid t, w(\cdot), \theta, u(\cdot))$ is the motion of system \cref{system}.
\end{theorem}
\begin{proof}
    The proof follows the standard scheme (see, e.g., \cite[Theorem~2.4.2]{Yong_2015}).
    Let $(t, w(\cdot)) \in G$ and $\theta \in [t, T]$ be fixed.
    Let us denote the right-hand side of the equality in (\ref{dynamic_programming_principle}) by $\bar{\Val}$ and prove that $\Val(t, w(\cdot)) = \bar{\Val}$.

    For every $u(\cdot) \in \mU(t, \theta)$ and $u^\ast(\cdot) \in \mU(\theta, T)$, let us consider $\bar{u}(\cdot) \in \mU(t, T)$ such that $\bar{u}(\tau) = u(\tau)$, $\tau \in [t, \theta)$, and $\bar{u}(\tau) = u^\ast(\tau)$, $\tau \in [\theta, T]$.
    Then, we have
    \begin{displaymath}
        \Val(t, w(\cdot))
        \leq J(t, w(\cdot), \bar{u}(\cdot))
        = J(\theta, x(\cdot), u^\ast(\cdot)) + \int_{t}^{\theta} \chi(\tau, x(\tau), u(\tau)) \, \rd \tau,
    \end{displaymath}
    where $x(\cdot) = x(\cdot \mid t, w(\cdot), \theta, u(\cdot))$.
    Taking the infimum over $u^\ast(\cdot) \in \mU(\theta, T)$, we obtain
    \begin{displaymath}
        \Val(t, w(\cdot))
        \leq \Val(\theta, x(\cdot)) + \int_{t}^{\theta} \chi(\tau, x(\tau), u(\tau)) \, \rd \tau.
    \end{displaymath}
    Further, taking the infimum over $u(\cdot) \in \mU(t, \theta)$, we get $\Val(t, w(\cdot)) \leq \bar{\Val}$.

    On the other hand, let $\varepsilon > 0$, and let $u^{(\varepsilon)}(\cdot) \in \mU(t, T)$ be an $\varepsilon$-optimal control (see \cref{varepsilon_optimal}).
    Let us define $u(\tau) = u^{(\varepsilon)}(\tau)$, $\tau \in [t, \theta]$, and $u^\ast(\tau) = u^{(\varepsilon)}(\tau)$, $\tau \in [\theta, T]$, and consider the motion $x(\cdot) = x(\cdot \mid t, w(\cdot), \theta, u(\cdot))$ of system \cref{system}.
    Hence, we derive
    \begin{multline*}
        \Val(t, w(\cdot)) + \varepsilon
        \geq J(t, w(\cdot), u^{(\varepsilon)}(\cdot))
        = J(\theta, x(\cdot), u^\ast(\cdot)) + \int_{t}^{\theta} \chi(\tau, x(\tau), u(\tau)) \, \rd \tau \\
        \geq \Val(\theta, x(\cdot)) + \int_{t}^{\theta} \chi(\tau, x(\tau), u(\tau)) \, \rd \tau
        \geq \bar{\Val}.
    \end{multline*}
    Since this estimate is valid for every $\varepsilon > 0$, then $\Val(t, w(\cdot)) \geq \bar{\Val}$.
    Thus, we have $\Val(t, w(\cdot)) = \bar{\Val}$, and the theorem is proved.
\end{proof}

The dynamic programming principle gives a nonlocal characterization of the value functional $\Val$.
However, in order to find this functional and, based on this, construct $\varepsilon$-optimal controls, it is often more convenient to apply an infinitesimal form of this principle, which is expressed in terms of a Hamilton--Jacobi--Bellman equation.
Before proceeding to the derivation and study of the Hamilton--Jacobi--Bellman equation associated with the optimal control problem \cref{system,cost_functional}, we establish some technical properties of motions of system \cref{system} and prove that the value functional $\Val$ is continuous.

\section{Properties of motions of the system}
\label{section_properties_of_motions}

The goal of this section is to prove the proposition below, which states the uniform boundedness and H\"{o}lder continuity of motions of system \cref{system}, and also their Lipschitz dependence on the initial positions.
\begin{proposition} \label{proposition_properties}
    For any $R > 0$, the following statements hold:
    \begin{itemize}
    \item[$i)$]
        there exists $M_x > 0$ such that, for any $(t, w(\cdot)) \in G$ satisfying $\|w(\cdot)\|_{[0, t]} \leq R$, any $\theta \in [t, T]$, and any $u(\cdot) \in \mU(t, \theta)$, the inequality $\|x(\cdot)\|_{[0, \theta]} \leq M_x$ is valid for the motion $x(\cdot) = x(\cdot \mid t, w(\cdot), \theta, u(\cdot))$ of system \cref{system};
    \item[$ii)$]
        there exists $H_x > 0$ such that, for any $(t, w(\cdot)) \in G$ satisfying $\|w(\cdot)\|_{[0, t]} \leq R$ and $\|(^C D^\alpha w)(\cdot)\|_{[0, t]} \leq R$, any $\theta \in [t, T]$, and any $u(\cdot) \in \mU(t, \theta)$, for the motion $x(\cdot) = x(\cdot \mid t, w(\cdot), \theta, u(\cdot))$ of system \cref{system}, the inequality below holds:
        \begin{equation} \label{H}
            \|x(\tau) - x(\tau^\prime)\|
            \leq H_x |\tau - \tau^\prime|^\alpha,
            \quad \tau, \tau^\prime \in [0, \theta];
        \end{equation}
    \item[$iii)$]
        there exists $L_x > 0$ such that, for any $(t, w(\cdot))$, $(t, w^\prime(\cdot)) \in G$ satisfying $\|w(\cdot)\|_{[0, t]} \leq R$ and $\|w^\prime(\cdot)\|_{[0, t]} \leq R$, any $\theta \in [t, T]$, and any $u(\cdot) \in \mU(t, \theta)$, for the motions $x(\cdot) = x(\cdot \mid t, w(\cdot), \theta, u(\cdot))$ and $x^\prime(\cdot) = x(\cdot \mid t, w^\prime(\cdot), \theta, u(\cdot))$ of system \cref{system}, the following inequality is valid:
        \begin{equation} \label{L}
            \|x(\cdot) - x^\prime(\cdot)\|_{[0, \theta]}
            \leq L_x \|w(\cdot) - w^\prime(\cdot)\|_{[0, t]}.
        \end{equation}
    \end{itemize}
\end{proposition}

The proof of this proposition is based on the following two lemmas.
\begin{lemma}\label{lemma_mean_value}
    For any $(t, w(\cdot)) \in G$, the inequality below holds:
    \begin{displaymath}
        \frac{1}{\Gamma(\alpha)} \Big\| \int_{0}^{t} \frac{(^C D^\alpha w)(\xi)}{(\tau - \xi)^{1 - \alpha}} \, \rd \xi \Big\|
        \leq \max_{\xi \in [0, t]} \|w(\xi) - w(0)\|,
        \quad \tau \in [t, T].
    \end{displaymath}
\end{lemma}

This lemma is proved by the scheme from \cite[Theorem~14.10 and Corollary~1]{Samko_Kilbas_Marichev_1993}.

\begin{lemma}[{\cite[Lemma~6.19]{Diethelm_2010}}] \label{lemma_Bellman_Gronwall}
    Let $t \in [0, T]$, $\theta \in [t, T]$, $\psi(\cdot) \in \C([t, \theta], \mathbb{R})$, $a \geq 0$, and $b \geq 0$ be such that
    \begin{displaymath}
        0 \leq
        \psi(\tau)
        \leq a + \frac{b}{\Gamma(\alpha)} \int_{t}^{\tau} \frac{\psi(\xi)}{(\tau - \xi)^{1 - \alpha}} \, \rd \xi,
        \quad \tau \in [t, \theta].
    \end{displaymath}
    Then, the following inequalities are valid:
    \begin{displaymath}
        \psi(\tau)
        \leq a E_\alpha((\tau - t)^\alpha b)
        \leq a E_\alpha(T^\alpha b),
        \quad \tau \in [t, \theta],
    \end{displaymath}
    where $E_\alpha$ is the Mittag-Leffler function of the order $\alpha$ {\rm(}see, e.g., {\rm \cite[(1.90)]{Samko_Kilbas_Marichev_1993})}.
\end{lemma}

\begin{proof}[Proof of \cref{proposition_properties}]
    Let $R > 0$ be fixed.

    Let us prove part $i)$.
    Taking $c_f$ from \cref{c_f}, we define $M_x = (1 + 3 R) E_\alpha(T^\alpha c_f) - 1$.
    Let $(t, w(\cdot)) \in G$ such that $\|w(\cdot)\|_{[0, t]} \leq R$, $\theta \in [t, T]$, and $u(\cdot) \in \mU(t, \theta)$ be fixed, and let $x(\cdot) = x(\cdot \mid t, w(\cdot), \theta, u(\cdot))$ be the corresponding motion of system \cref{system}.
    Due to \cref{integral_equation,lemma_mean_value}, we derive
    \begin{multline*}
            \|x(\tau)\|
            \leq \|w(0)\|
            + \frac{1}{\Gamma(\alpha)} \Big\| \int_{0}^{t} \frac{( {}^C D^\alpha w) (\xi)}{(\tau - \xi)^{1 - \alpha}} \, \rd \xi \Big\|
            + \frac{1}{\Gamma(\alpha)} \Big\| \int_{t}^{\tau} \frac{f(\xi, x(\xi), u(\xi))}{(\tau - \xi)^{1 - \alpha}} \, \rd \xi \Big\| \\
            \leq \|w(0)\|
            + \max_{\xi \in [0, t]} \|w(\xi) - w(0)\|
            + \frac{c_f}{\Gamma(\alpha)} \int_{t}^{\tau} \frac{1 + \|x(\xi)\|}{(\tau - \xi)^{1 - \alpha}} \, \rd \xi,
            \quad \tau \in [t, \theta].
    \end{multline*}
    Hence,
    \begin{displaymath}
        1 + \|x(\tau)\|
        \leq 1 + 3 R + \frac{c_f}{\Gamma(\alpha)} \int_{t}^{\tau} \frac{1 + \|x(\xi)\|}{(\tau - \xi)^{1 - \alpha}} \, \rd \xi,
        \quad \tau \in [t, \theta],
    \end{displaymath}
    wherefrom, applying \cref{lemma_Bellman_Gronwall}, we obtain $\|x(\tau)\| \leq M_x$ for $\tau \in [t, \theta]$.
    Since $M_x \geq R$, then this estimate also holds for $\tau \in [0, t)$.
    Thus, $\|x(\cdot)\|_{[0, \theta]} \leq M_x$.

    Further, we prove part $ii)$.
    Let $H_x = H_\alpha \max\{R, (1 + M_x) c_f\}$, where $H_\alpha$ is taken from \cref{H_alpha}.
    Let $(t, w(\cdot)) \in G$ such that $\|w(\cdot)\|_{[0, t]} \leq R$ and $\|(^C D^\alpha w)(\cdot)\|_{[0, t]} \leq R$, $\theta \in [t, T]$, and $u(\cdot) \in \mU(t, \theta)$ be fixed, and let $x(\cdot) = x(\cdot \mid t, w(\cdot), \theta, u(\cdot))$.
    We have
    \begin{displaymath}
        \|(^C D^\alpha x)(\tau)\|
        = \|(^C D^\alpha w)(\tau)\|
        \leq R
        \text{ for a.e. } \tau \in [0, t].
    \end{displaymath}
    Moreover, due to \cref{system,c_f}, we obtain
    \begin{displaymath}
        \|(^C D^\alpha x)(\tau)\|
        = \|f(\tau, x(\tau), u(\tau))\|
        \leq (1 + \|x(\tau)\|) c_f
        \leq (1 + M_x) c_f
        \text{ for a.e. } \tau \in [t, \theta].
    \end{displaymath}
    Then, $\|(^C D^\alpha x)(\cdot)\|_{[0, \theta]} \leq \max\{R, (1 + M_x) c_f\}$, and \cref{H} follows from \cref{ID,H_alpha}.

    Finally, let us prove part $iii)$.
    According to local Lipschitz continuity of $f$ in $x$ (see \cref{Assumption_f}), let us choose $\lambda_f$ by $M_x$ and define $L_x = 3 E_\alpha(T^\alpha \lambda_f)$.
    Let $(t, w(\cdot)),$ $(t, w^\prime(\cdot)) \in G$ such that $\|w(\cdot)\|_{[0, t]} \leq R$ and $\|w^\prime(\cdot)\|_{[0, t]} \leq R$, $\theta \in [t, T]$, and $u(\cdot) \in \mU(t, \theta)$ be fixed.
    Let us consider the motions $x(\cdot) = x(\cdot \mid t, w(\cdot), \theta, u(\cdot))$ and $x^\prime(\cdot) = x(\cdot \mid t, w^\prime(\cdot), \theta, u(\cdot))$ of system \cref{system}.
    Due to \cref{integral_equation}, we derive
    \begin{multline*}
        \|x(\tau) - x^\prime(\tau)\|
        \leq \|w(0) - w^\prime(0)\|
        + \frac{1}{\Gamma(\alpha)} \Big\| \int_{0}^{t} \frac{({}^C D^\alpha w)(\xi) - ({}^C D^\alpha w^\prime)(\xi)}{(\tau - \xi)^{1 - \alpha}} \, \rd \xi \Big\| \\
        + \frac{1}{\Gamma(\alpha)} \Big\| \int_{t}^{\tau} \frac{f(\xi, x(\xi), u(\xi)) - f(\xi, x^\prime(\xi), u(\xi))}{(\tau - \xi)^{1 - \alpha}} \, \rd \xi \Big\|,
        \quad \tau \in [t, \theta].
    \end{multline*}
    Let us denote $\bar{w}(\tau) = w(\tau) - w^\prime(\tau)$, $\tau \in [0, t]$.
    Then, $\bar{w}(\cdot) \in \AC([0, t], \mathbb{R}^n)$ and, consequently, $(t, \bar{w}(\cdot)) \in G$.
    Moreover, in accordance with \cref{Caputo_derivative}, we have $({}^C D^\alpha \bar{w})(\tau) = ({}^C D^\alpha w)(\tau) - ({}^C D^\alpha w^\prime)(\tau)$ for a.e. $\tau \in [0, t]$.
    Therefore, applying \cref{lemma_mean_value}, we get
    \begin{displaymath}
        \frac{1}{\Gamma(\alpha)} \Big\| \int_{0}^{t} \frac{({}^C D^\alpha w)(\xi) - ({}^C D^\alpha w^\prime)(\xi)}{(\tau - \xi)^{1 - \alpha}} \, \rd \xi \Big\|
        \leq \max_{\xi \in [0, t]} \|\bar{w}(\xi) - \bar{w}(0) \|
        \leq 2 \|w(\cdot) - w^\prime(\cdot)\|_{[0, t]},
    \end{displaymath}
    where $\tau \in [t, \theta]$.
    Hence,
    \begin{displaymath}
        \|x(\tau) - x^\prime(\tau)\|
        \leq 3 \|w(\cdot) - w^\prime(\cdot)\|_{[0, t]}
        + \frac{\lambda_f}{\Gamma(\alpha)} \int_{t}^{\tau} \frac{\|x(\xi) - x^\prime(\xi)\|}{(\tau - \xi)^{1 - \alpha}} \, \rd \xi,
        \quad \tau \in [t, \theta],
    \end{displaymath}
    wherefrom, by \cref{lemma_Bellman_Gronwall}, we obtain $\|x(\tau) - x^\prime(\tau)\| \leq L_x \|w(\cdot) - w^\prime(\cdot)\|_{[0, t]}$ for $\tau \in [t, \theta]$.
    Since $L_x \geq 1$, then this estimate also holds for $\tau \in [0, t)$, and \cref{L} is proved.
\end{proof}

Let us emphasize that, in part $ii)$ of \cref{proposition_properties}, the H\"{o}lder constant $H_x$ of the motion $x(\cdot)$ depends not only on the estimate of the history $w(\cdot)$, but also on the estimate of its derivative $(^C D^\alpha w)(\cdot)$.

\section{Continuity of the value functional}
\label{section_continuity}

First, let us prove the uniform continuity of the value functional $\Val: G \rightarrow \mathbb{R}$ (see \cref{Val}) in the functional argument $w(\cdot)$.
\begin{lemma} \label{lemma_continuity_fixed_t}
    For any $R > 0$ and $\varepsilon > 0$, there exists $\delta > 0$ such that, for any $(t, w(\cdot))$, $(t, w^\prime(\cdot)) \in G$, if $\|w(\cdot)\|_{[0, t]} \leq R$ and $\|w(\cdot) - w^\prime(\cdot)\|_{[0, t]} \leq \delta$, then
    \begin{equation} \label{lemma_continuity_fixed_t_main}
        |\Val(t, w(\cdot)) - \Val(t, w^\prime(\cdot))|
        \leq \varepsilon.
    \end{equation}
\end{lemma}
\begin{proof}
    Let $R > 0$ and $\varepsilon > 0$ be fixed.
    By $R^\prime = R + 1$, let us choose $M_x^\prime$ and $L_x^\prime$ according to \cref{proposition_properties}.
    Due to continuity of $\sigma$ and $\chi$ (see \cref{Assumption_sigma_chi}), there exists $\delta \in (0, 1]$ such that, for any $x,$ $x^\prime \in B(M_x^\prime)$ satisfying $\|x - x^\prime\| \leq L_x^\prime \delta$, any $\tau \in [0, T]$, and any $u \in P$, the following inequalities are valid:
    \begin{displaymath}
        |\sigma(x) - \sigma(x^\prime)|
        \leq \varepsilon / 2,
        \quad
        |\chi(\tau, x, u) - \chi(\tau, x^\prime, u)|
        \leq \varepsilon / (2 T).
    \end{displaymath}
    Let us show that the statement of the lemma is valid for this $\delta$.

    Let $(t, w(\cdot)),$ $(t, w^\prime(\cdot)) \in G$ be fixed, and let $\|w(\cdot)\|_{[0, t]} \leq R$, $\|w(\cdot) - w^\prime(\cdot)\|_{[0, t]} \leq \delta$.
    Let us note that $\|w^\prime(\cdot)\|_{[0, t]} \leq R^\prime$.
    For every $u(\cdot) \in \mU(t, T)$, we consider the motions $x(\cdot) = x(\cdot \mid t, w(\cdot), T, u(\cdot))$ and $x^\prime(\cdot) = x(\cdot \mid t, w^\prime(\cdot), T, u(\cdot))$ of system \cref{system}.
    We have $\|x(\cdot)\|_{[0, T]} \leq M_x^\prime$, $\|x^\prime(\cdot)\|_{[0, T]} \leq M_x^\prime$, and $\|x(\cdot) - x^\prime(\cdot)\|_{[0, T]} \leq L_x^\prime \delta$.
    Hence,
    \begin{multline*}
        |J(t, w(\cdot), u(\cdot)) - J(t, w^\prime(\cdot), u(\cdot))| \\
        \leq |\sigma(x(T)) - \sigma(x^\prime(T))|
        + \int_{t}^{T} |\chi(\tau, x(\tau), u(\tau)) - \chi(\tau, x^\prime(\tau), u(\tau))| \, \rd \tau
        \leq \varepsilon.
    \end{multline*}
    Since this estimate is valid for every $u(\cdot) \in \mU(t, T)$, then, in accordance with \cref{Val}, we conclude \cref{lemma_continuity_fixed_t_main}.
    The lemma is proved.
\end{proof}

In order to study the continuity of $\Val$ with respect to both variables $t$ and $w(\cdot)$, the set $G$ is endowed with the following metric (see, e.g., \cite[p.~25]{Lukoyanov_2011_Eng} and also \cite{Lukoyanov_2003}):
\begin{multline} \label{dist}
    \dist \big((t, w(\cdot)), (t^\prime, w^\prime(\cdot))\big) \\
    = \max \Big\{ \dist^\ast\big((t, w(\cdot)), (t^\prime, w^\prime(\cdot))\big),
    \dist^\ast\big((t^\prime, w^\prime(\cdot)), (t, w(\cdot))\big) \Big\},
\end{multline}
where $(t, w(\cdot))$, $(t^\prime, w^\prime(\cdot)) \in G$,
\begin{displaymath}
    \dist^\ast\big((t, w(\cdot)), (t^\prime, w^\prime(\cdot))\big)
    = \max_{\tau \in [0, t]} \min_{\tau^\prime \in [0, t^\prime]}
    \big( |\tau - \tau^\prime|^2 + \|w(\tau) - w^\prime(\tau^\prime)\|^2 \big)^{1/2},
\end{displaymath}
and $\dist^\ast((t^\prime, w^\prime(\cdot)), (t, w(\cdot)))$ is defined in a similar way with clear changes.
Let us note that this metric is a Hausdorff distance between the graphics of the functions $w: [0, t] \rightarrow \mathbb{R}^n$ and $w^\prime: [0, t^\prime] \rightarrow \mathbb{R}^n$.

The proposition below holds (see also \cite[(P2.12)]{Lukoyanov_2011_Eng}).
\begin{proposition} \label{proposition_dist}
    For any $(t, w(\cdot))$, $(t^\prime, w^\prime(\cdot)) \in G$, $t^\prime \leq t$, the inequalities
    \begin{align}
        & \dist
        \leq t - t^\prime + \varkappa(t - t^\prime) + \|w_{t^\prime}(\cdot) - w^\prime(\cdot)\|_{[0, t^\prime]}, \label{proposition_dist_1} \\
        & t - t^\prime
        \leq \dist,
        \quad \|w_{t^\prime}(\cdot) - w^\prime(\cdot)\|_{[0, t^\prime]}
        \leq \dist + \varkappa(\dist) \label{proposition_dist_2}
    \end{align}
    are valid, where $\dist = \dist((t, w(\cdot)), (t^\prime, w^\prime(\cdot)))$, $\varkappa$ is the modulus of continuity of $w(\cdot)$, and $w_{t^\prime}(\tau) = w(\tau)$, $\tau \in [0, t^\prime]$.
\end{proposition}
\begin{proof}
    Let $(t, w(\cdot))$, $(t^\prime, w^\prime(\cdot)) \in G$ such that $t^\prime \leq t$ be fixed.

    For every $\tau \in [0, t]$, if $\tau \leq t^\prime$, then
    \begin{displaymath}
        \min_{\tau^\prime \in [0, t^\prime]}
        \big( |\tau - \tau^\prime|^2 + \|w(\tau) - w^\prime(\tau^\prime)\|^2 \big)^{1/2}
        \leq \|w(\tau) - w^\prime(\tau)\|
        \leq \|w_{t^\prime}(\cdot) - w^\prime(\cdot)\|_{[0, t^\prime]},
    \end{displaymath}
    and, if $\tau > t^\prime$, then
    \begin{multline*}
        \min_{\tau^\prime \in [0, t^\prime]}
        \big( |\tau - \tau^\prime|^2 + \|w(\tau) - w^\prime(\tau^\prime)\|^2 \big)^{1/2}
        \leq \tau - t^\prime + \|w(\tau) - w^\prime(t^\prime)\| \\
        \leq t - t^\prime + \varkappa(t - t^\prime) + \|w_{t^\prime}(\cdot) - w^\prime(\cdot)\|_{[0, t^\prime]}.
    \end{multline*}
    Therefore, we obtain
    \begin{displaymath}
        \dist^\ast\big((t, w(\cdot)), (t^\prime, w^\prime(\cdot))\big)
        \leq t - t^\prime + \varkappa(t - t^\prime) + \|w_{t^\prime}(\cdot) - w^\prime(\cdot)\|_{[0, t^\prime]}.
    \end{displaymath}
    On the other hand, for every $\tau^\prime \in [0, t^\prime]$, we have
    \begin{displaymath}
        \min_{\tau \in [0, t]}
        \big( |\tau - \tau^\prime|^2 + \|w(\tau) - w^\prime(\tau^\prime)\|^2 \big)^{1/2}
        \leq \|w(\tau^\prime) - w^\prime(\tau^\prime)\|
        \leq \|w_{t^\prime}(\cdot) - w^\prime(\cdot)\|_{[0, t^\prime]}.
    \end{displaymath}
    Thus,
    \begin{displaymath}
        \dist^\ast\big((t^\prime, w^\prime(\cdot)), (t, w(\cdot))\big)
        \leq \|w_{t^\prime}(\cdot) - w^\prime(\cdot)\|_{[0, t^\prime]},
    \end{displaymath}
    and we conclude \cref{proposition_dist_1}.

    Further, the first inequality in \cref{proposition_dist_2} follows from the estimates
    \begin{displaymath}
        \dist
        \geq \dist^\ast\big((t, w(\cdot)), (t^\prime, w^\prime(\cdot))\big)
        \geq \min_{\tau^\prime \in [0, t^\prime]}
        \big( |t - \tau^\prime|^2 + \|w(t) - w^\prime(\tau^\prime)\|^2 \big)^{1/2}
        \geq t - t^\prime.
    \end{displaymath}

    Let us prove the second inequality in \cref{proposition_dist_2}.
    Let us choose $\tau^\prime_0 \in [0, t^\prime]$ from the condition $\|w_{t^\prime}(\cdot) - w^\prime(\cdot)\|_{[0, t^\prime]} = \|w(\tau^\prime_0) - w^\prime(\tau^\prime_0)\|$ and take $\tau_0 \in [0, t]$ such that
    \begin{displaymath}
        \min_{\tau \in [0, t]}
        \big( |\tau - \tau^\prime_0|^2 + \|w(\tau) - w^\prime(\tau^\prime_0)\|^2 \big)
        = |\tau_0 - \tau^\prime_0|^2 + \|w(\tau_0) - w^\prime(\tau^\prime_0)\|^2.
    \end{displaymath}
    Hence, since
    \begin{displaymath}
        \dist
        \geq \dist^\ast\big((t^\prime, w^\prime(\cdot)), (t, w(\cdot))\big)
        \geq \big( |\tau_0 - \tau^\prime_0|^2 + \|w(\tau_0) - w^\prime(\tau^\prime_0)\|^2 \big)^{1/2},
    \end{displaymath}
    then $|\tau_0 - \tau^\prime_0| \leq \dist$, $\|w(\tau_0) - w^\prime(\tau^\prime_0)\| \leq \dist$, and, consequently,
    \begin{displaymath}
        \|w(\tau^\prime_0) - w^\prime(\tau^\prime_0)\|
        \leq \|w(\tau_0) - w^\prime(\tau^\prime_0)\| + \|w(\tau^\prime_0) - w(\tau_0)\|
        \leq \dist + \varkappa(\dist).
    \end{displaymath}
    Therefore, due to the choice of $\tau^\prime_0$, we get the second inequality in \cref{proposition_dist_2}.
\end{proof}

Now, let us prove that the value functional $\Val: G \rightarrow \mathbb{R}$ (see \cref{Val}) is continuous.
\begin{theorem} \label{theorem_continuity}
    For any $(t, w(\cdot)) \in G$ and $\varepsilon > 0$, there exists $\delta > 0$ such that, for any $(t^\prime, w^\prime(\cdot)) \in G$, from the inequality $\dist((t, w(\cdot)), (t^\prime, w^\prime(\cdot))) \leq \delta$, it follows that
    \begin{equation} \label{theorem_continuity_main}
        |\Val(t, w(\cdot)) - \Val(t^\prime, w^\prime(\cdot))|
        \leq \varepsilon.
    \end{equation}
\end{theorem}
\begin{proof}
    Let $(t, w(\cdot)) \in G$ and $\varepsilon > 0$ be fixed.
    Let $R > 0$ be such that $\|w(\cdot)\|_{[0, t]} \leq R$ and $\|(^C D^\alpha w)(\cdot)\|_{[0, t]} \leq R$, and let $M_x$ and $H_x$ be chosen according to \cref{proposition_properties}.
    Due to continuity of $\chi$ (see \cref{Assumption_sigma_chi}), there exists $M_\chi > 0$ such that
    \begin{equation} \label{M_chi}
        |\chi(\tau, x, u)|
        \leq M_\chi,
        \quad \tau \in [0, T], \quad x \in B(M_x), \quad u \in P.
    \end{equation}
    Let $\delta_1 > 0$ be such that $M_\chi \delta_1 \leq \varepsilon / 6$.
    Further, applying \cref{lemma_continuity_fixed_t} for $M_x$ and $\varepsilon / 3$, let us choose $\delta_2$.
    Finally, taking $c_f$ from \cref{c_f}, we consider $\delta_3 > 0$ such that
    \begin{displaymath}
        \big((1 + M_x) c_f + R\big) \delta_3^\alpha / \Gamma(\alpha + 1)
        \leq \delta_2,
        \quad 2 \delta_3 + H_x \delta_3^\alpha + H_x \big( 2 \delta_3 + H_x \delta_3^\alpha \big)^\alpha
        \leq \delta_2.
    \end{displaymath}
    Let us show that the statement of the theorem is valid for $\delta = \min\{\delta_1, \delta_3\}$.

    Let $(t^\prime, w^\prime(\cdot)) \in G$ be fixed such that $\dist((t, w(\cdot)), (t^\prime, w^\prime(\cdot))) \leq \delta$.
    Let us note that $|t - t^\prime| \leq \delta$ due to \cref{proposition_dist}.

    Let us consider the case when $t^\prime \leq t$.
    Taking into account that $w(\cdot)$ is H\"{o}lder continuous of the order $\alpha$ with the constant $H_x$, applying \cref{proposition_dist}, we derive $\|w_{t^\prime}(\cdot) - w^\prime(\cdot)\|_{[0, t^\prime]} \leq \delta + H_x \delta^\alpha \leq \delta_2$.
    Since $\|w_{t^\prime}(\cdot)\|_{[0, t^\prime]} \leq M_x$, then, by \cref{lemma_continuity_fixed_t},
    \begin{equation} \label{proof_Theprem_continuity_1_1}
        | \Val(t^\prime, w_{t^\prime}(\cdot)) - \Val(t^\prime, w^\prime(\cdot)) |
        \leq \varepsilon / 3.
    \end{equation}
    According to \cref{Theorem_dynamic_programming_principle}, there exists $u^\prime(\cdot) \in \mU(t^\prime, t)$ such that, for the corrsponding motion $x^\prime(\cdot) = x(\cdot \mid t^\prime, w_{t^\prime}(\cdot), t, u^\prime(\cdot))$ of system \cref{system}, we have
    \begin{displaymath}
        \Big| \Val(t^\prime, w_{t^\prime}(\cdot))
        - \Val(t, x^\prime(\cdot)) - \int_{t^\prime}^{t} \chi(\tau, x^\prime(\tau), u^\prime(\tau)) \, \rd \tau \Big|
        \leq \varepsilon / 6.
    \end{displaymath}
    Hence, taking into account that $\|x^\prime(\cdot)\|_{[0, t]} \leq M_x$, we get
    \begin{equation} \label{proof_Theprem_continuity_1_2}
        | \Val(t^\prime, w_{t^\prime}(\cdot)) - \Val(t, x^\prime(\cdot)) |
        \leq \varepsilon / 3.
    \end{equation}
    Due to \cref{integral_equation,ID}, applying \cref{c_f}, we derive
    \begin{multline*}
        \|x^\prime(\tau) - w(\tau)\|
        \leq \frac{1}{\Gamma(\alpha)} \int_{t^\prime}^{\tau} \frac{\|f(\xi, x^\prime(\xi), u^\prime(\xi))\| + \|(^C D^\alpha w)(\xi)\|}{(\tau - \xi)^{1 - \alpha}} \, \rd \xi \\
        \leq \frac{(1 + M_x) c_f + R}{\Gamma(\alpha + 1)} (\tau - t^\prime)^\alpha
        \leq \delta_2,
        \quad \tau \in [t^\prime, t].
    \end{multline*}
    Thus, since $x^\prime(\tau) = w(\tau)$, $\tau \in [0, t^\prime]$, then $\|x^\prime(\cdot) - w(\cdot)\|_{[0, t]} \leq \delta_2$, and, by \cref{lemma_continuity_fixed_t},
    \begin{equation} \label{proof_Theprem_continuity_1_3}
        | \Val(t, x^\prime(\cdot)) - \Val(t, w(\cdot)) |
        \leq \varepsilon / 3.
    \end{equation}
    From \cref{proof_Theprem_continuity_1_1,proof_Theprem_continuity_1_2,proof_Theprem_continuity_1_3}, we conclude \cref{theorem_continuity_main}.

    Now, let us suppose that $t^\prime > t$.
    By \cref{Theorem_dynamic_programming_principle}, there exists $u(\cdot) \in \mU(t, t^\prime)$ such that, for the motion $x(\cdot) = x(\cdot \mid t, w(\cdot), t^\prime, u(\cdot))$ of system \cref{system}, we have
    \begin{displaymath}
        \Big| \Val(t, w(\cdot)) - \Val(t^\prime, x(\cdot)) - \int_{t}^{t^\prime} \chi(\tau, x(\tau), u(\tau)) \, \rd \tau \Big|
        \leq \varepsilon / 2,
    \end{displaymath}
    and, consequently, taking into account that $\|x(\cdot)\|_{[0, t^\prime]} \leq M_x$, we get
    \begin{equation} \label{proof_Theprem_continuity_2_1}
        |\Val(t, w(\cdot)) - \Val(t^\prime, x(\cdot))|
        \leq 2 \varepsilon / 3.
    \end{equation}
    Since $x(\cdot)$ is H\"{o}lder continuous of the order $\alpha$ with the constant $H_x$ and $x(\tau) = w(\tau)$, $\tau \in [0, t]$, then, applying \cref{proposition_dist}, we derive $\dist((t^\prime, x(\cdot)), (t, w(\cdot))) \leq \delta + H_x \delta^\alpha$.
    Hence, by the triangle inequality,
    \begin{multline*}
        \dist\big((t^\prime, x(\cdot)), (t^\prime, w^\prime(\cdot))\big) \\
        \leq \dist\big((t^\prime, x(\cdot)), (t, w(\cdot))\big)
        + \dist\big((t, w(\cdot)), (t^\prime, w^\prime(\cdot))\big)
        \leq 2 \delta + H_x \delta^\alpha.
    \end{multline*}
    Thus, using \cref{proposition_dist} again, we obtain
    \begin{displaymath}
        \|x(\cdot) - w^\prime(\cdot)\|_{[0, t^\prime]}
        \leq 2 \delta + H_x \delta^\alpha + H_x \big( 2 \delta + H_x \delta^\alpha \big)^\alpha
        \leq \delta_2.
    \end{displaymath}
    Therefore, according to \cref{lemma_continuity_fixed_t}, we conclude
    \begin{equation} \label{proof_Theprem_continuity_2_2}
        |\Val(t^\prime, x(\cdot)) - \Val(t^\prime, w^\prime(\cdot))|
        \leq \varepsilon / 3.
    \end{equation}
    From \cref{proof_Theprem_continuity_2_1,proof_Theprem_continuity_2_2}, we derive \cref{theorem_continuity_main}.
    The theorem is proved.
\end{proof}

\section{Coinvariant derivatives of fractional order}
\label{section_derivatives}

In order to write out the Hamilton--Jacobi--Bellman equation associated with the considered optimal control problem \cref{system,cost_functional}, let us introduce an appropriate notion of derivatives of functionals defined on the set $G$ of positions of system \cref{system}.

Let us denote $G^0 = \{(t, w(\cdot)) \in G: \, t < T\}$, and, for every $(t, w(\cdot)) \in G^0$, in accordance with \cref{initial_condition_general}, consider the set of admissible extensions $x(\cdot)$ of $w(\cdot)$:
\begin{equation} \label{X}
    X(t, w(\cdot))
    = \big\{ x(\cdot) \in \AC([0, T], \mathbb{R}^n):
    \, x(\tau) = w(\tau), \, \tau \in [0, t] \big\}.
\end{equation}

A functional $\varphi: G \rightarrow \mathbb{R}$ is called coinvariantly ($ci$-) differentiable of the order $\alpha$ at $(t, w(\cdot)) \in G^0$, if there exist $\partial_t^\alpha \varphi(t, w(\cdot)) \in \mathbb{R}$ and $\nabla^\alpha \varphi(t, w(\cdot)) \in \mathbb{R}^n$ such that, for every $x(\cdot) \in X(t, w(\cdot))$ and $\tau \in (t, T)$, the relation below is valid:
\begin{multline} \label{derivatives_definition}
    \varphi(\tau, x_\tau(\cdot)) - \varphi(t, w(\cdot))
    = \partial_t^\alpha \varphi(t, w(\cdot)) (\tau - t) \\
    + \big\langle \nabla^\alpha \varphi(t, w(\cdot)),
    \big(I^{1 - \alpha} (x(\cdot) - x(0))\big)(\tau) - \big(I^{1 - \alpha} (w(\cdot) - w(0)) \big)(t) \big\rangle \\
    + o(\tau - t),
\end{multline}
where $x_\tau(\cdot)$ is determined by $x(\cdot)$ according to \cref{x_t}, $o(\tau - t)$ may depend on $x(\cdot)$, and $o(\tau - t)/(\tau - t) \rightarrow 0$ when $\tau \downarrow t$.
Respectively, $\partial_t^\alpha \varphi(t, w(\cdot))$ and $\nabla^\alpha \varphi(t, w(\cdot))$ are called the $ci$-derivative in $t$ and the $ci$-gradient of the order $\alpha$ of $\varphi$ at $(t, w(\cdot))$.
Let us note that, due to \cref{int_D^alpha,X}, relation \cref{derivatives_definition} can be rewritten as follows:
\begin{multline} \label{derivatives_definition_alternative}
    \varphi(\tau, x_\tau(\cdot)) - \varphi(t, w(\cdot)) \\
    = \partial_t^\alpha \varphi(t, w(\cdot)) (\tau - t)
    + \big\langle \nabla^\alpha \varphi(t, w(\cdot)), \int_{t}^{\tau} (^C D^\alpha x)(\xi) \, \rd \xi \big\rangle
    + o(\tau - t).
\end{multline}

\begin{remark}
    In the case $\alpha = 1$, the introduced notion of $ci$-differentiability of the order $\alpha$ agrees with the notion of $ci$-differentiability developed, e.g., in \cite[\S~2]{Lukoyanov_2011_Eng} and \cite[\S~2.4]{Kim_1999} (see also \cite{Kim_1985,Lukoyanov_2000_JAMM,Lukoyanov_2003}).
    Indeed, in this case, the set $X(t, w(\cdot))$ consists of all Lipschitz continuous extensions of $w(\cdot)$, and \cref{derivatives_definition} becomes
    \begin{displaymath}
        \varphi(\tau, x_\tau(\cdot)) - \varphi(t, w(\cdot))
        = \partial_t^1 \varphi(t, w(\cdot)) (\tau - t)
        + \langle \nabla^1 \varphi(t, w(\cdot)), x(\tau) - w(t) \rangle
        + o(\tau - t).
    \end{displaymath}
\end{remark}

Let us note that \cref{derivatives_definition} determines the derivatives $\partial_t^\alpha \varphi(t, w(\cdot))$ and $\nabla^\alpha \varphi(t, w(\cdot))$ uniquely.
In order to verify this, one should take an arbitrary $l \in \mathbb{R}^n$, consider the extension $x^{(l)}(\cdot) \in X(t, w(\cdot))$ such that $(^C D^\alpha x)(\tau) = l$, $\tau \in (t, T]$, substitute it in \cref{derivatives_definition} (or, equivalently, in \cref{derivatives_definition_alternative}), and take into account the invariance of $\partial_t^\alpha \varphi(t, w(\cdot))$ and $\nabla^\alpha \varphi(t, w(\cdot))$ with respect to $x^{(l)}(\cdot)$ (see also the arguments in \cite[p.~34]{Lukoyanov_2011_Eng}).

Further, a functional $\varphi: G \rightarrow \mathbb{R}$ is called $ci$-smooth of the order $\alpha$, if it satisfies the following conditions:
(a) $\varphi$ is continuous;
(b) $\varphi$ is $ci$-differentiable of the order $\alpha$ at every $(t, w(\cdot)) \in G^0$;
(c) the functionals $\partial_t^\alpha \varphi: G^0 \rightarrow \mathbb{R}$ and $\nabla^\alpha \varphi: G^0 \rightarrow \mathbb{R}^n$ are continuous.
We recall that the set $G^0 \subset G$ is endowed with the metric $\dist$ (see \cref{dist}).

The notion of fractional $ci$-derivatives allows us to obtain a simple formula for the total derivative of a $ci$-smooth functional along motions of system \cref{system}.
\begin{lemma} \label{lemma_formula}
    Let a functional $\varphi: G \rightarrow \mathbb{R}$ be $ci$-smooth of the order $\alpha$.
    Then, for any $(t, w(\cdot)) \in G^0$, $x(\cdot) \in X(t, w(\cdot))$, and $\theta \in [t, T)$, the function $\omega(\tau) = \varphi(\tau, x_\tau(\cdot))$, $\tau \in [t, T]$, is continuous on $[t, T]$ and Lipschitz continuous on $[t, \theta]$.
    Moreover,
    \begin{equation} \label{main_Lem_formula}
        \frac{\rd \omega(\tau)}{\rd \tau}
        = \partial_t^\alpha \varphi(\tau, x_\tau(\cdot)) + \langle \nabla^\alpha \varphi(\tau, x_\tau(\cdot)), (^C D^\alpha x)(\tau) \rangle
        \text{ for a.e. } \tau \in [t, \theta].
    \end{equation}
\end{lemma}
\begin{proof}
    We follow the scheme from \cite[Lemma~2.1]{Lukoyanov_2011_Eng} (see also \cite[Lemma~3]{Plaksin_2019_IFAC}).

    Let $(t, w(\cdot)) \in G^0$, $x(\cdot) \in X(t, w(\cdot)$, and $\theta \in [t, T)$ be fixed.
    Since $x(\cdot)$ is continuous, then, by \cref{proposition_dist}, the function $[t, T] \ni \tau \mapsto (\tau, x_\tau(\cdot)) \in G$ is continuous.
    Hence, due to continuity of $\varphi$, we obtain that $\omega(\cdot)$ is continuous on $[t, T]$.
    Moreover, since the functionals $\partial_t^\alpha \varphi$ and $\nabla^\alpha \varphi$ are also continuous, there exists $M > 0$ such that $|\partial_t^\alpha \varphi(\tau, x_\tau(\cdot))| \leq M$ and $\|\nabla^\alpha \varphi(\tau, x_\tau(\cdot))\| \leq M$ for $\tau \in [t, \theta]$.
    Let us define $L = (1 + \|(^C D^\alpha x)(\cdot)\|_{[t, T]}) M$ and prove the inequality
    \begin{equation} \label{proof_Lem_formula_omega}
        |\omega(\tau) - \omega(\tau^\prime)|
        \leq L |\tau - \tau^\prime|,
        \quad \tau, \tau^\prime \in [t, \theta].
    \end{equation}
    Let $\tau \in [t, \theta]$ be fixed.
    Since $\varphi$ is $ci$-differentiable of the order $\alpha$ at $(\tau, x_\tau(\cdot))$ and $x(\cdot) \in X(\tau, x_\tau(\cdot))$, then, according to \cref{derivatives_definition_alternative}, for every $\delta \in (0, T - \tau)$, we have
    \begin{displaymath}
        \frac{\omega(\tau + \delta) -  \omega(\tau)}{\delta}
        = \partial_t^\alpha \varphi(\tau, x_\tau(\cdot))
        + \big\langle \nabla^\alpha \varphi(\tau, x_\tau(\cdot)), \frac{1}{\delta} \int_{\tau}^{\tau + \delta} (^C D^\alpha x)(\xi) \, \rd \xi \big\rangle
        + \frac{o(\delta)}{\delta},
    \end{displaymath}
    and, hence, $|\omega(\tau + \delta) - \omega(\tau)| / \delta \leq L + |o(\delta)|/\delta$.
    Thus, we get
    \begin{displaymath}
        \Big|\limsup_{\delta \downarrow 0} \frac{\omega(\tau + \delta) -  \omega(\tau)}{\delta} \Big|
        \leq L.
    \end{displaymath}
    Taking into account that this inequality hods for every $\tau \in [t, \theta]$, and $\omega(\cdot)$ is continuous on $[t, \theta]$, by Dini's theorem (see, e.g., \cite[Ch.~4, Theorem~1.2]{Bruckner_1978}), we conclude \cref{proof_Lem_formula_omega}.

    Further, let us prove \cref{main_Lem_formula}.
    Namely, let us show that the equality in \cref{main_Lem_formula} is valid for every $\tau \in (t, \theta)$ such that the derivatives $(^C D^\alpha x)(\tau)$ and $\rd \omega (\tau) / \rd \tau$ exist.
    According to \cref{derivatives_definition,Caputo_derivative}, we derive
    \begin{multline*}
        \frac{\rd \omega(\tau)}{\rd \tau}
        = \lim_{\delta \downarrow 0} \frac{\omega(\tau + \delta) - \omega(\tau)}{\delta}
        = \partial_t^\alpha \varphi(\tau, x_\tau(\cdot)) \\
        + \big\langle \nabla^\alpha \varphi(\tau, x_\tau(\cdot)),
        \lim_{\delta \downarrow 0} \frac{\big(I^{1 - \alpha} (x(\cdot) - x(0))\big)(\tau + \delta)
        - \big(I^{1 - \alpha} (x(\cdot) - x(0)) \big)(\tau)}{\delta} \big\rangle \\
        = \partial_t^\alpha \varphi(\tau, x_\tau(\cdot))
        + \langle \nabla^\alpha \varphi(\tau, x_\tau(\cdot)), (^C D^\alpha x)(\tau) \rangle.
    \end{multline*}
    The lemma is proved.
\end{proof}

\begin{remark}
    It seems that the proposed notion of fractional $ci$-differentiation may also be used for the needs of the stability theory for fractional-order systems.
    Namely, the absence of a simple rule for calculating the fractional derivative of the composition of two functions (see, e.g., \cite{Tarasov_2016} and the references therein) leads to the difficulties in defining the notion of the total derivative of a Lyapunov function $V(t, x)$ along motions of the system (see, e.g. \cite{Agarwal_O'Regan_Hristova_2015}).
    In particular, it becomes important and relevant to obtain different estimates for the fractional derivative of the composition $V(t,x(t))$ of the function $V(t, x)$ and a motion $x(\cdot)$ (see, e.g., \cite{Alikhanov_2010,Aguila-Camacho_Duarte-Mermoud_Gallegos_2014,Gomoyunov_2019_FCAA,Tuan_Trinh_2018} and the references therein).
    However, according to the discussion in \cref{Section_Example_1} (see also, e.g., \cite{Burton_2011}), it seems that, by analogy with functional-differential systems (see \cite{Krasovskii_1963} and also \cite{Kim_1999}), for studying the questions of stability for fractional-order systems, it may be useful to apply Lyapunov--Krasovskii functionals $V(t, w(\cdot))$ that depend on the history $w(\cdot)$ of a motion $x(\cdot)$.
    In this case, \cref{lemma_formula} gives a simple formula for the total derivative of $V(t, w(\cdot))$ along motions of the system.
\end{remark}

\section{Hamilton--Jacobi--Bellman equation}
\label{section_HJB}

With the optimal control problem \cref{system,cost_functional}, let us associate the Hamilton--Jacobi--Bellman equation
\begin{equation} \label{HJB}
    \partial_t^\alpha \varphi(t, w(\cdot))
    + \H\big( t, w(t), \nabla^\alpha \varphi(t, w(\cdot)) \big)
    = 0,
    \quad (t, w(\cdot)) \in G^0,
\end{equation}
where $\partial_t^\alpha \varphi(t, w(\cdot))$ and $\nabla^\alpha \varphi(t, w(\cdot))$ are the $ci$-derivatives of the order $\alpha$ of $\varphi$ at $(t, w(\cdot))$ (see \cref{section_derivatives}), and the Hamiltonian is defined in the usual way:
\begin{equation} \label{Hamiltonian}
    \H(\tau, x, s)
    = \min_{u \in P} \big( \langle s, f(\tau, x, u) \rangle + \chi(\tau, x, u)\big),
    \quad \tau \in [0, T], \quad x, s \in \mathbb{R}^n.
\end{equation}

The main result of this section is the following.
\begin{theorem} \label{theorem_differentiability}
    If the value functional $\Val$ {\rm(}see \cref{Val}{\rm)} is $ci$-smooth, then it satisfies Hamilton--Jacobi--Bellman equation \cref{HJB}.
\end{theorem}

Let us note that, since $\Val$ is always continuous by \cref{theorem_continuity}, then the essential assumptions of \cref{theorem_differentiability} are $ci$-differentiability of the order $\alpha$ of $\rho$ at every $(t, w(\cdot)) \in G^0$ and continuity of the functionals $\partial_t^\alpha \Val: G^0 \rightarrow \mathbb{R}$ and $\nabla^\alpha \Val: G^0 \rightarrow \mathbb{R}^n$.

In order to prove \cref{theorem_differentiability}, and also \cref{theorem_varphi_rho} below, we need some auxiliary constructions.

For every $(t, w(\cdot)) \in G^0$, taking $c_f$ from \cref{c_f}, let us consider the sets
\begin{multline}\label{X_ast}
    X_\ast(t, w(\cdot))
    = \big\{ x(\cdot) \in X(t, w(\cdot)): \\
    \|(^C D^\alpha x)(\tau)\| \leq (1 + \|x(\tau)\|) c_f \text{ for a.e. } \tau \in [t, T] \big\},
\end{multline}
where $X(t, w(\cdot))$ is defined according to \cref{X}, and
\begin{displaymath}
    G_\ast(t, w(\cdot))
    = \big\{ (\tau, x_\tau(\cdot)) \in G:
    \, \tau \in [0, T], \, x(\cdot) \in X_\ast(t, w(\cdot)) \big\}.
\end{displaymath}

In the proposition below, some properties of these sets are given.
\begin{proposition} \label{proposition_XG}
    For any $(t, w(\cdot)) \in G^0$, the following statements are valid:
    \begin{itemize}
        \item[$i)$]
            for any $u(\cdot) \in \mU(t, T)$, the motion $x(\cdot) = x(\cdot \mid t, w(\cdot), T, u(\cdot))$ of system \cref{system} satisfies the inclusion $x(\cdot) \in X_\ast(t, w(\cdot))$;
            therefore, in particular, we have $(\tau, x_\tau(\cdot)) \in G_\ast(t, w(\cdot))$, $\tau \in [0, T]$;
        \item[$ii)$]
            there exist $R_\ast > 0$ and $H_\ast > 0$ such that, for any $x(\cdot) \in X_\ast(t, w(\cdot)),$ the inequalities below hold:
            \begin{displaymath}
                \|x(\cdot)\|_{[0, T]} \leq R_\ast,
                \quad \|x(\tau) - x(\tau^\prime)\|
                \leq H_\ast |\tau - \tau^\prime|^\alpha,
                \quad \tau, \tau^\prime \in [0, T];
            \end{displaymath}
        \item[$iii)$]
            the set $X_\ast(t, w(\cdot))$ is a compact subset of $\AC([0, T], \mathbb{R}^n)$;
        \item[$iv)$]
            the set $G_\ast(t, w(\cdot))$ is a compact subset of $G$.
    \end{itemize}
\end{proposition}
\begin{proof}
    Part $i)$ follows from definition \cref{X_ast} of $X_\ast(t, w(\cdot))$ and the sublinear growth of $f$ in $x$ (see \cref{Assumption_f}).
    The proof of $ii)$ is carried out by analogy with parts $i)$ and $ii)$ of \cref{proposition_properties}.
    Part $iii)$ is a consequence of \cite[Assertion~7]{Gomoyunov_2019_Trudy_Eng}.

    Let us prove part $iv)$.
    Let $(t, w(\cdot)) \in G^0$ and a sequence $\{(t_m, w_m(\cdot))\}_{m \in \mathbb{N}} \subset G_\ast(t, w(\cdot))$ be fixed.
    For every $m \in \mathbb{N}$, let us choose $x^{(m)}(\cdot) \in X_\ast(t, w(\cdot))$ such that $x_{t_m}^{(m)}(\cdot) = w_m(\cdot)$.
    According to part $iii)$, there exist $\theta \in [0, T]$, $x_\ast(\cdot) \in X_\ast(t, w(\cdot))$ and subsequences $\{t_{m_i}\}_{i \in \mathbb{N}}$, $\{x^{(m_i)}(\cdot)\}_{i \in \mathbb{N}}$ such that $t_{m_i} \rightarrow \theta$ and $\|x_\ast(\cdot) - x^{(m_i)}(\cdot)\|_{[0, T]} \rightarrow 0$ when $i \rightarrow \infty$.
    Let us consider the function $w_\ast(\tau) = x_\ast(\tau),$ $\tau \in [0, \theta]$.
    The inclusion $x_\ast(\cdot) \in X_\ast(t, w(\cdot))$ implies $(\theta, w_\ast(\cdot)) \in G_\ast(t, w(\cdot))$.
    Applying \cref{proposition_dist} (see also \cite[Lemma~1.1]{Lukoyanov_2011_Eng}), we derive $\dist((\theta, w_\ast(\cdot)), (t_{m_i}, w_{m_i}(\cdot))) \rightarrow 0$ when $i \rightarrow \infty$, and, therefore, the subsequence $\{(t_{m_i}, w_{m_i}(\cdot))\}_{i \in \mathbb{N}}$ converges to $(\theta, w_\ast(\cdot)) \in G_\ast(t, w(\cdot))$.
    Thus, the set $G_\ast(t, w(\cdot))$ is a compact subset of $G$.
    The proposition is proved.
\end{proof}

\begin{proof}[Proof of \cref{theorem_differentiability}]
    Let us note that \cref{lemma_formula} allows us to prove the theorem by applying the standard arguments (see, e.g., \cite[Proposition~2.4.3]{Yong_2015}).

    Let $(t, w(\cdot)) \in G^0$ be fixed, and let $R_\ast$ and $H_\ast$ be chosen according to \cref{proposition_XG}.
    Let us take $\theta \in (t, T)$ and introduce the set
    \begin{displaymath}
        G^\theta_\ast(t, w(\cdot))
        = \big\{ (t^\prime, w^\prime(\cdot)) \in G_\ast(t, w(\cdot)):
        \, t^\prime \leq \theta \big\}.
    \end{displaymath}
    It follows from part~$iv)$ of \cref{proposition_XG} that $G^\theta_\ast(t, w(\cdot))$ is a compact subset of $G^0$.
    Therefore, since the functional $\nabla^\alpha \Val: G^0 \rightarrow \mathbb{R}^n$ is continuous, there exists $M > 0$ such that $\|\nabla^\alpha \Val(t^\prime, w^\prime(\cdot))\| \leq M$ for $(t^\prime, w^\prime(\cdot)) \in G^\theta_\ast(t, w(\cdot))$.

    Further, let us fix $\varepsilon > 0$.
    Due to continuity of the functionals $\partial_t^\alpha \Val: G^0 \rightarrow \mathbb{R}$ and $\nabla^\alpha \Val$, let us choose $\delta_1 > 0$ such that, for any $(t^\prime, w^\prime(\cdot))$, $(t^{\prime \prime}, w^{\prime \prime}(\cdot)) \in G^\theta_\ast(t, w(\cdot))$ satisfying $\dist((t^\prime, w^\prime(\cdot)),(t^{\prime \prime}, w^{\prime \prime}(\cdot))) \leq \delta_1 + H_\ast \delta_1^\alpha$, the inequality below holds:
    \begin{displaymath}
        |\partial_t^\alpha \Val(t^\prime, w^\prime(\cdot)) - \partial_t^\alpha \Val(t^{\prime \prime}, w^{\prime \prime}(\cdot))|
        + \|\nabla^\alpha \Val(t^\prime, w^\prime(\cdot)) - \nabla^\alpha \Val(t^{\prime \prime}, w^{\prime \prime}(\cdot))\| (1 + R_\ast) c_f
        \leq \varepsilon / 2,
    \end{displaymath}
    where $c_f$ is taken from \cref{c_f}.
    Moreover, since $f$ and $\chi$ are continuous (see \cref{Assumption_f,Assumption_sigma_chi}), there exists $\delta_2 > 0$ such that, for any $\tau,$ $\tau^\prime \in [0, T]$, $x$, $x^\prime \in B(R_\ast)$ satisfying $|\tau - \tau^\prime| \leq \delta_2$, $\|x - x^\prime\| \leq H_\ast \delta_2^\alpha$, and any $u \in P$, we get
    \begin{displaymath}
        M \|f(\tau, x, u) - f(\tau^\prime, x^\prime, u)\|
        + |\chi(\tau, x, u) - \chi(\tau^\prime, x^\prime, u)|
        \leq \varepsilon / 2.
    \end{displaymath}
    Let us define $\delta = \min\{\delta_1, \delta_2, \theta - t\}$.
    Hence, according to \cref{proposition_dist,proposition_XG}, we obtain the following property.
    Let $u(\cdot) \in \mU(t, T)$, and let $x(\cdot) = x(\cdot \mid t, w(\cdot), T, u(\cdot))$ be the motion of system \cref{system}.
    Then, for every $u^\prime \in P$, denoting
    \begin{equation} \label{mu}
        \mu(\tau)
        = \partial_t^\alpha \Val(\tau, x_\tau(\cdot))
        + \langle \nabla^\alpha \Val(\tau, x_\tau(\cdot)), f(\tau, x(\tau), u^\prime) \rangle
        + \chi(\tau, x(\tau), u^\prime)
    \end{equation}
    for $\tau \in [t, \theta]$, we have $|\mu(\tau) - \mu(\tau^\prime)| \leq \varepsilon$ for any $\tau,$ $\tau^\prime \in [t, \theta]$ such that $|\tau - \tau^\prime | \leq \delta$.

    Now, let us prove that
    \begin{equation} \label{theorem_differentiability_proof_1}
        \partial_t^\alpha \Val(t, w(\cdot))
        + \H\big(t, w(t), \nabla^\alpha \Val(t, w(\cdot)) \big)
        \geq 0.
    \end{equation}
    For every $u \in P$, let us consider the constant control $u(\tau) = u$, $\tau \in [t, t + \delta]$, and the corresponding motion $x(\cdot) = x(\cdot \mid t, w(\cdot), t + \delta, u(\cdot))$ of system \cref{system}.
    By \cref{Theorem_dynamic_programming_principle},
    \begin{equation} \label{theorem_differentiability_proof_1_1}
        \Val(t + \delta, x(\cdot))
        + \int_{t}^{t + \delta} \chi(\tau, x(\tau), u) \, \rd \tau
        - \Val(t, w(\cdot)) \geq 0.
    \end{equation}
    Formally extending the motion $x(\cdot)$ up to $T$ and applying \cref{lemma_formula}, we obtain that the function $\omega(\tau) = \Val(\tau, x_\tau(\cdot)),$ $\tau \in [t, t + \delta]$, is Lipschitz continuous and
    \begin{displaymath}
        \frac{\rd \omega(\tau)}{\rd \tau}
        = \partial_t^\alpha \Val(\tau, x_\tau(\cdot))
        + \langle \nabla^\alpha \Val(\tau, x_\tau(\cdot)), f(\tau, x(\tau), u) \rangle
        \text{ for a.e. } \tau \in [t, t + \delta].
    \end{displaymath}
    Therefore, we have
    \begin{multline} \label{theorem_differentiability_proof_1_2}
            \Val(t + \delta, x(\cdot)) - \Val(t, w(\cdot))
            = \omega(t + \delta) - \omega(t)
            = \int_{t}^{t + \delta} \frac{\rd \omega(\tau)}{\rd \tau} \, \rd \tau \\
            = \int_{t}^{t + \delta} \Big( \partial_t^\alpha \Val(\tau, x_\tau(\cdot))
            + \langle \nabla^\alpha \Val (\tau, x_\tau(\cdot)), f(\tau, x(\tau), u) \rangle \Big) \, \rd \tau.
    \end{multline}
    From \cref{theorem_differentiability_proof_1_1,theorem_differentiability_proof_1_2}, we derive
    \begin{displaymath}
        \int_{t}^{t + \delta} \Big( \partial_t^\alpha \Val (\tau, x_\tau(\cdot))
        + \langle \nabla^\alpha \Val (\tau, x_\tau(\cdot)), f(\tau, x(\tau), u) \rangle + \chi(\tau, x(\tau), u) \Big) \, \rd \tau
        \geq 0.
    \end{displaymath}
    Hence, the choice of $\delta$ yields (see \cref{mu})
    \begin{displaymath}
        \partial_t^\alpha \Val (t, w(\cdot))
        + \langle \nabla^\alpha \Val (t, w(\cdot)), f(t, w(t), u) \rangle + \chi(t, w(t), u)
        \geq - \varepsilon.
    \end{displaymath}
    Since this estimate holds for every $u \in P$ and $\varepsilon > 0$, then, due to \cref{Hamiltonian}, we get \cref{theorem_differentiability_proof_1}.

    On the other hand, let us show that
    \begin{equation} \label{theorem_differentiability_proof_2}
        \partial_t^\alpha \Val(t, w(\cdot))
        + \H \big(t, w(t), \nabla^\alpha \Val(t, w(\cdot)) \big)
        \leq 0.
    \end{equation}
    According to \cref{Theorem_dynamic_programming_principle}, there exists $u(\cdot) \in \mU(t, t + \delta)$ such that
    \begin{displaymath}
        \Val(t + \delta, x(\cdot))
        + \int_{t}^{t + \delta} \chi(\tau, x(\tau), u(\tau)) \, \rd \tau
        - \Val(t, w(\cdot))
        \leq \varepsilon \delta,
    \end{displaymath}
    where $x(\cdot) = x(\cdot \mid t, w(\cdot), t + \delta, u(\cdot))$.
    Then, arguing as above, we obtain
    \begin{displaymath}
        \int_{t}^{t + \delta} \Big( \partial_t^\alpha \Val(\tau, x_\tau(\cdot))
        + \langle \nabla^\alpha \Val(\tau, x_\tau(\cdot)), f(\tau, x(\tau), u(\tau)) \rangle
        + \chi(\tau, x(\tau), u(\tau)) \Big) \, \rd \tau
        \leq \varepsilon \delta
    \end{displaymath}
    and, by the choice of $\delta$,
    \begin{displaymath}
        \int_{t}^{t + \delta} \Big( \partial_t^\alpha \Val(t, w(\cdot))
        + \langle \nabla^\alpha \Val(t, w(\cdot)), f(t, w(t), u(\tau)) \rangle
        + \chi(t, w(t), u(\tau)) \Big) \, \rd \tau
        \leq 2 \varepsilon \delta.
    \end{displaymath}
    Hence, taking \cref{Hamiltonian} into account, we derive
    \begin{displaymath}
        \partial_t^\alpha \Val (t, w(\cdot))
        + \H\big( t, w(t), \nabla^\alpha \Val (t, w(\cdot)) \big)
        \leq 2 \varepsilon.
    \end{displaymath}
    Since this inequality is valid for every $\varepsilon > 0$, then we get \cref{theorem_differentiability_proof_2}.

    From \cref{theorem_differentiability_proof_1,theorem_differentiability_proof_2}, we conclude \cref{HJB}.
    The theorem is proved.
\end{proof}

\section{Optimal control strategy}
\label{section_optimal_control_strategy}

In this section, we establish the result that, in some sense, is converse to \cref{theorem_differentiability}.
Namely, we prove that a $ci$-smooth of the order $\alpha$ functional $\varphi: G \rightarrow \mathbb{R}$ that satisfies Hamilton--Jacobi--Bellman equation \cref{HJB} and the natural right-end condition (see \cref{Val_T})
\begin{equation}\label{terminal_condition}
    \varphi(T, w(\cdot))
    = \sigma(w(T)),
    \quad w(\cdot) \in \AC([0, T], \mathbb{R}^n),
\end{equation}
coincides with the value functional $\Val$ (see \cref{Val}).
Furthermore, we propose a way of forming $\varepsilon$-optimal controls on the basis of this functional $\varphi$ by using a stepwise feedback control scheme.
In order to formulate this result, we consider the following formalization of such control schemes, which goes back to the positional approach in differential games \cite{Krasovskii_Subbotin_1988,Krasovskii_Krasovskii_1995} (see also \cite{Osipov_1971,Lukoyanov_2011_Eng,Lukoyanov_2000_JAMM,Lukoyanov_2003} and \cite{Gomoyunov_2019_Trudy_Eng,Gomoyunov_2019_DGAA}).

By a control strategy, we mean an arbitrary functional $U: G^0 \rightarrow P$.
Let us fix a position $(t, w(\cdot)) \in G^0$ and a partition of the time interval $[t, T]$:
\begin{equation} \label{Delta}
    \Delta
    = \{\tau_j\}_{j \in \overline{1, k}},
    \quad \tau_1 = t,
    \quad \tau_{j + 1} > \tau_j,
    \quad j \in \overline{1, k},
    \quad \tau_{k + 1} = T,
    \quad k \in \mathbb{N}.
\end{equation}
The pair $\{U, \Delta\}$ is called a control law.
This control law forms a piecewise constant control $u(\cdot) \in \mU(t, T)$ and the corresponding motion $x(\cdot) = x(\cdot \mid t, w(\cdot), T, u(\cdot))$ of system \cref{system} according to the following recursive procedure:
at every time $\tau_j$, $j \in \overline{1, k}$, the history $x_{\tau_j}(\cdot)$ (see \cref{x_t}) of the motion $x(\cdot)$ on $[0, \tau_j]$ is measured, the value $u_j = U(\tau_j, x_{\tau_j}(\cdot))$ is computed, and then the constant control $u(\tau) = u_j$ is applied until $\tau_{j + 1}$, when a new measurement of the history is taken.
In a short form, we have
\begin{equation} \label{control_law}
    u(\tau)
    = U(\tau_j, x_{\tau_j}(\cdot)),
    \quad \tau \in [\tau_j, \tau_{j + 1}),
    \quad j \in \overline{1, k}.
\end{equation}
Formally putting $u(T) = \tilde{u}$ for some fixed $\tilde{u} \in P$, we conclude that the described procedure determines $u(\cdot)$ and $x(\cdot)$ uniquely.
For the obtained control $u(\cdot)$, we also use the notation $u(\cdot \mid t, w(\cdot), U, \Delta)$.
Let us note that, according to \cref{Val}, the corresponding value of cost functional \cref{cost_functional_general} satisfies the estimate
\begin{displaymath}
    J\big(t, w(\cdot), u(\cdot \mid t, w(\cdot), U, \Delta)\big)
    \geq \Val(t, w(\cdot)).
\end{displaymath}

Taking this into account, we call a control strategy $U^\circ$ optimal if the following statement holds.
For any $(t, w(\cdot)) \in G^0$ and $\varepsilon > 0$, there exists $\delta > 0$ such that, for any partition $\Delta$ (see \cref{Delta}) with the diameter $\diam(\Delta) = \max_{j \in \overline{1, k}} (\tau_{j + 1} - \tau_j) \leq \delta$, the inequality below is valid:
\begin{displaymath}
    J\big(t, w(\cdot), u(\cdot \mid t, w(\cdot), U^\circ, \Delta)\big)
    \leq \Val(t, w(\cdot)) + \varepsilon,
\end{displaymath}
or, in other words, the control $u(\cdot \mid t, w(\cdot), U^\circ, \Delta)$ is $\varepsilon$-optimal.
Thus, the problem of constructing $\varepsilon$-optimal controls can be reduced to finding an optimal control strategy~$U^\circ$.

The main result of this section is the following.
\begin{theorem} \label{theorem_varphi_rho}
    Let a $ci$-smooth of the order $\alpha$ functional $\varphi: G \rightarrow \mathbb{R}$ satisfies Hamilton--Jacobi--Bellman equation \cref{HJB} and right-end condition \cref{terminal_condition}.
    Then,
    \begin{equation} \label{theorem_varphi_rho_main}
        \varphi(t, w(\cdot))
        = \Val(t, w(\cdot)),
        \quad (t, w(\cdot)) \in G,
    \end{equation}
    and the control strategy
    \begin{equation} \label{U^0}
        U^\circ(t, w(\cdot))
        \in \arg \min_{u \in P} \big( \langle \nabla^\alpha \varphi(t, w(\cdot)), f(t, w(t), u) \rangle + \chi(t, w(t), u) \big),
    \end{equation}
    where $(t, w(\cdot)) \in G^0$, is optimal.
\end{theorem}

The proof of the theorem follows the scheme from, e.g., \cite[Theorem~3.1]{Lukoyanov_2011_Eng} (see also \cite[Theorem~3.1]{Lukoyanov_2003} and \cite[Theorem~1]{Gomoyunov_Plaksin_2018_IFAC}), and, for convenience, is divided into the following two lemmas, which are valid under the assumptions of the theorem.

\begin{lemma} \label{lemma_first}
    For any $(t, w(\cdot)) \in G^0$, the inequality below holds:
    \begin{equation} \label{Lem_leq_main}
        \Val(t, w(\cdot))
        \geq \varphi(t, w(\cdot)).
    \end{equation}
\end{lemma}
\begin{proof}
    Let us fix $(t, w(\cdot)) \in G^0$, take $u(\cdot) \in \mU(t, T)$, and show that
    \begin{equation} \label{proof_Lem_leq_main}
        J(t, w(\cdot), u(\cdot))
        \geq \varphi(t, w(\cdot)).
    \end{equation}

    Let us consider the function
    \begin{equation} \label{omega}
        \omega(\tau)
        = \varphi(\tau, x_\tau(\cdot)) + \int_{t}^{\tau} \chi(\xi, x(\xi), u(\xi)) \, \rd \xi,
        \quad \tau \in [t, T],
    \end{equation}
    where $x(\cdot) = x(\cdot \mid t, w(\cdot), T, u(\cdot))$ is the motion of system \cref{system}.
    Let $\varepsilon > 0$ be fixed.
    Due to \cref{lemma_formula} and continuity of $\chi$, the function $\omega(\cdot)$ is continuous, and, therefore, one can choose $\theta \in (t, T)$ such that
    \begin{equation} \label{proof_Lem_leq_main_first}
        \omega(T) - \omega(\theta)
        \geq - \varepsilon.
    \end{equation}
    Moreover, the function $\omega(\cdot)$ is Lipschitz continuous on $[t, \theta]$ and, for a.e. $\tau \in [t, \theta]$,
    \begin{equation} \label{omega_dot}
        \frac{\rd \omega(\tau)}{\rd \tau}
        = \partial_t^\alpha \varphi (\tau, x_\tau(\cdot))
        + \langle \nabla^\alpha \varphi (\tau, x_\tau(\cdot)), f(\tau, x(\tau), u(\tau)) \rangle
        + \chi(\tau, x(\tau), u(\tau)).
    \end{equation}
    Further, according to \cref{Hamiltonian,HJB}, we have
    \begin{multline*}
            \partial_t^\alpha \varphi (\tau, x_\tau(\cdot))
            + \langle \nabla^\alpha \varphi (\tau, x_\tau(\cdot)), f(\tau, x(\tau), u(\tau)) \rangle
            + \chi(\tau, x(\tau), u(\tau)) \\
            \geq \partial_t^\alpha \varphi (\tau, x_\tau(\cdot))
            + \H \big(\tau, x(\tau), \nabla^\alpha \varphi (\tau, x_\tau(\cdot)) \big)
            = 0,
            \quad t \in [t, \theta].
    \end{multline*}
    Consequently, we derive $\rd \omega(t) / \rd t \geq 0$ for a.e. $\tau \in [t, \theta]$, and, hence, $\omega(\theta) - \omega(t) \geq 0$.
    Thus, it follows from \cref{proof_Lem_leq_main_first} that $\omega(T) - \omega(t) \geq - \varepsilon$.
    Since this inequality holds for every $\varepsilon > 0$, then $\omega(T) - \omega(t) \geq 0$.
    Taking into account that, due to \cref{terminal_condition,omega},
    \begin{equation} \label{omega_t_T}
        \omega(t) = \varphi(t, w(\cdot)),
        \quad
        \omega(T)
        = J(t, w(\cdot), u(\cdot)),
    \end{equation}
    the obtained estimate yields \cref{proof_Lem_leq_main}.

    Since \cref{proof_Lem_leq_main} is valid for every $u(\cdot) \in \mU(t, T)$, then, in accordance with \cref{Val}, we get \cref{Lem_leq_main}.
    The lemma is proved.
\end{proof}

\begin{lemma} \label{lemma_second}
    For any $(t, w(\cdot)) \in G^0$ and $\varepsilon > 0$, there exists $\delta > 0$ such that, for any partition $\Delta$ {\rm(}see \cref{Delta}{\rm)} satisfying $\diam(\Delta) \leq \delta$, the inequality below holds:
    \begin{displaymath}
        J\big(t, w(\cdot), u(\cdot \mid t, w(\cdot), U^\circ, \Delta)\big)
        \leq \varphi(t, w(\cdot)) + \varepsilon.
    \end{displaymath}
\end{lemma}
\begin{proof}
    Let $(t, w(\cdot)) \in G^0$ and $\varepsilon > 0$ be fixed.
    Let $R_\ast$ and $H_\ast$ be taken from \cref{proposition_XG}.
    Due to continuity of $\varphi: G \rightarrow \mathbb{R}$ and compactness of $G_\ast(t, w(\cdot))$ (see \cref{proposition_XG}), let us choose $\eta_1 > 0$ such that, for any $(t^\prime, w^\prime(\cdot))$, $(t^{\prime \prime}, w^{\prime \prime}(\cdot)) \in G_\ast(t, w(\cdot))$ satisfying $\dist((t^\prime, w^\prime(\cdot)), (t^{\prime \prime}, w^{\prime \prime}(\cdot))) \leq \eta_1 + H_\ast \eta_1^\alpha$, we have
    \begin{displaymath}
        |\varphi(t^\prime, w^\prime(\cdot)) - \varphi(t^{\prime \prime}, w^{\prime \prime}(\cdot))|
        \leq \varepsilon / 4.
    \end{displaymath}
    Since $\chi$ is continuous, then there exists $M_\chi > 0$ such that \cref{M_chi} is valid where $R_\ast$ is substituted instead of $M_x$.
    Let us choose $\eta_2 > 0$ from the condition $M_\chi \eta_2 \leq \varepsilon / 4$ and take $\theta \in (t, T)$ such that $T - \theta \leq \min\{\eta_1, \eta_2\}$.
    Then, according to \cref{proposition_dist,proposition_XG}, we obtain the following property.
    Let $u(\cdot) \in \mU(t, T)$, and let $x(\cdot) = x(\cdot \mid t, w(\cdot), T, u(\cdot))$ be the corresponding motion of system \cref{system}.
    Then,
    \begin{equation} \label{lemma_second_proof_eta}
        \varphi(T, x(\cdot)) + \int_{\theta}^{T} \chi(\tau, x(\tau), u(\tau)) \, \rd \tau - \varphi(\theta, x_\theta(\cdot))
        \leq \varepsilon / 2.
    \end{equation}
    Further, arguing as in the proof of \cref{theorem_differentiability}, let us choose $\delta > 0$ such that, if $u(\cdot) \in \mU(t, T)$, $x(\cdot) = x(\cdot \mid t, w(\cdot), T, u(\cdot))$, and $u^\prime \in P$, then the function $\mu(\cdot)$ defined by \cref{mu} where the functional $\varphi$ is substituted instead of $\Val$ satisfies the inequality $|\mu(\tau) - \mu(\tau^\prime)| \leq \varepsilon / (2 (\theta - t))$ for any $\tau,$ $\tau^\prime \in [t, \theta]$ satisfying $|\tau - \tau^\prime | \leq \delta$.
    Let us show that the statement of the lemma is valid for the chosen $\delta$.

    Let us fix a partition $\Delta$ such that $\diam(\Delta) \leq \delta$ and consider the control $u(\cdot) = u(\cdot \mid t, w(\cdot), U^\circ, \Delta)$ and the motion $x(\cdot) = x(\cdot \mid t, w(\cdot), T, u(\cdot))$ of system \cref{system} formed by the control law $\{U^\circ, \Delta\}$.
    As in the proof of \cref{lemma_first}, let us define the function $\omega(\cdot)$ by \cref{omega}.
    Due to \cref{lemma_second_proof_eta}, we have $\omega(T) - \omega(\theta) \leq \varepsilon / 2$.
    Therefore, according to \cref{omega_t_T}, in order to complete the proof, it is sufficient to verify that $\omega(\theta) - \omega(t) \leq \varepsilon / 2$.
    Hence, by \cref{omega_dot}, it remains to show that, for every $\tau \in [t, \theta)$,
    \begin{multline} \label{lemma_second_proof_inequality}
        \partial_t^\alpha \varphi (\tau, x_\tau(\cdot))
        + \langle \nabla^\alpha \varphi (\tau, x_\tau(\cdot)), f(\tau, x(\tau), u(\tau)) \rangle
        + \chi(\tau, x(\tau), u(\tau)) \\
        \leq \varepsilon / (2 (\theta - t)).
    \end{multline}

    Let $\tau \in [t, \theta)$, and let $j \in \overline{1, k}$ be such that $\tau \in [\tau_j, \tau_{j + 1})$.
    By the choice of $\delta$,
    \begin{multline*}
        \partial_t^\alpha \varphi (\tau, x_\tau(\cdot))
        + \langle \nabla^\alpha \varphi (\tau, x_\tau(\cdot)), f(\tau, x(\tau), u(\tau)) \rangle
        + \chi(\tau, x(\tau), u(\tau)) \\
        \leq \partial_t^\alpha \varphi (\tau_j, x_{\tau_j}(\cdot))
        + \langle \nabla^\alpha \varphi (\tau_j, x_{\tau_j}(\cdot)), f(\tau_j, x(\tau_j), u(\tau)) \rangle
        + \chi(\tau_j, x(\tau_j), u(\tau)) \\
        + \varepsilon / (2 (\theta - t)).
    \end{multline*}
    Since, from \cref{control_law,U^0,Hamiltonian}, it follows that $u(\tau) = u_j = U^\circ(\tau_j, x_{\tau_j}(\cdot))$ and
    \begin{displaymath}
        \langle \nabla^\alpha \varphi (\tau_j, x_{\tau_j}(\cdot)), f(\tau_j, x(\tau_j), u_j) \rangle
        + \chi(\tau_j, x(\tau_j), u_j)
        = \H \big( \tau_j, x(\tau_j), \nabla^\alpha \varphi(\tau_j, x_{\tau_j}(\cdot)) \big),
    \end{displaymath}
    then, due to \cref{HJB}, we obtain
    \begin{multline*}
        \partial_t^\alpha \varphi (\tau_j, x_{\tau_j}(\cdot))
        + \langle \nabla^\alpha \varphi (\tau_j, x_{\tau_j}(\cdot)), f(\tau_j, x(\tau_j), u_j) \rangle
        + \chi(\tau_j, x(\tau_j), u_j) \\
        = \partial_t^\alpha \varphi (\tau_j, x_{\tau_j}(\cdot))
        + \H \big( \tau_j, x(\tau_j), \nabla^\alpha \varphi(\tau_j, x_{\tau_j}(\cdot)) \big)
        = 0.
    \end{multline*}
    Thus, \cref{lemma_second_proof_inequality} holds, and the lemma is proved.
\end{proof}

\begin{proof}[Proof of \cref{theorem_varphi_rho}]
    Let $(t, w(\cdot)) \in G$ be fixed.
    If $t = T$, then \cref{theorem_varphi_rho_main} is a consequence of \cref{terminal_condition,Val_T}.
    Let us suppose that $t < T$.
    Then, due to \cref{lemma_second}, in accordance with \cref{Val}, we get $\Val(t, w(\cdot)) \leq \varphi(t, w(\cdot))$.
    Hence, applying \cref{lemma_first}, we conclude \cref{theorem_varphi_rho_main}.
    Finally, the optimality of the strategy $U^\circ$ follows directly from \cref{theorem_varphi_rho_main,lemma_second}.
    The theorem is proved.
\end{proof}

From \cref{theorem_continuity,theorem_differentiability,theorem_varphi_rho,Val_T}, we derive the following.
\begin{corollary} \label{Corollary_Final}
    If the value functional $\Val$ {\rm(}see \cref{Val}{\rm)} is $ci$-smooth of the order $\alpha$, then the control strategy defined by $\Val$ according to \cref{U^0} is optimal.
\end{corollary}

Thus, this corollary gives a way of constructing an optimal control strategy in the case when the value functional $\Val$ satisfies the additional smoothness assumptions.
Let us note that, in the general case, the desired control strategy can be constructed on the basis of the methods developed in \cite{Gomoyunov_2019_Trudy_Eng,Gomoyunov_2019_DGAA}.

\section{Example: complete solution}
\label{Section_Example_2}

Let us illustrate the results obtained in the paper by solving the example considered in \cref{Section_Example_1}.
Namely, let us calculate the value functional $\rho$ and find an optimal control strategy $U^\circ$ in the optimal control problem for the dynamical system
\begin{displaymath}
    (^C D^\alpha x)(\tau)
    = \Gamma(\alpha + 1) u(\tau),
    \quad x(\tau) \in \mathbb{R}, \quad |u(\tau)| \leq 1, \quad \tau \in [0, T],
\end{displaymath}
and the cost functional
\begin{displaymath}
    J(t, w(\cdot), u(\cdot))
    = x^2(T \mid t, w(\cdot), T, u(\cdot)).
\end{displaymath}

First, let us consider the auxiliary functional
\begin{equation} \label{example_Val_ast}
    \Val_\ast(t, w(\cdot))
    = w(0) + \frac{1}{\Gamma(\alpha)} \int_{0}^{t} \frac{(^C D^\alpha w)(\xi)}{(T - \xi)^{1 - \alpha}} \, \rd \xi,
    \quad (t, w(\cdot)) \in G.
\end{equation}
Let us note that, in accordance with \cref{integral_equation}, $\Val_\ast$ is the value functional of the degenerate optimal control problem for the system $(^C D^\alpha x)(\tau) = 0$, $x(\tau) \in \mathbb{R}$, $\tau \in [0, T]$, and the cost functional $J(t, w(\cdot), u(\cdot)) = x(T \mid t, w(\cdot), T, u(\cdot))$.
Therefore, $\Val_\ast$ is continuous by \cref{theorem_continuity}.
Further, let us fix $(t, w(\cdot)) \in G^0$ and show that $\Val_\ast$ is $ci$-differentiable of the order $\alpha$ at $(t, w(\cdot))$.
For every $x(\cdot) \in X(t, w(\cdot))$ and $\tau \in (t, T)$, we have
\begin{displaymath}
    \Val_\ast(\tau, x_\tau(\cdot)) - \Val_\ast(t, w(\cdot))
    = \frac{1}{\Gamma(\alpha)} \int_{t}^{\tau} \frac{(^C D^\alpha x)(\xi)}{(T - \xi)^{1 - \alpha}} \, \rd \xi.
\end{displaymath}
Let us denote
\begin{displaymath}
    z(\xi)
    = \int_{t}^{\xi} (^C D^\alpha x)(\zeta) \, \rd \zeta,
    \quad \xi \in [t, \tau],
    \quad \bar{H} = \|(^C D^\alpha x)(\cdot)\|_{[0, T]}.
\end{displaymath}
Then, applying the integration by parts formula, we obtain
\begin{displaymath}
    \int_{t}^{\tau} \frac{(^C D^\alpha x)(\xi)}{(T - \xi)^{1 - \alpha}} \, \rd \xi
    = \frac{z(\tau)}{(T - \tau)^{1 - \alpha}}
    - (1 - \alpha) \int_{t}^{\tau} \frac{z(\xi)}{(T - \xi)^{2 - \alpha}} \, \rd \xi.
\end{displaymath}
Hence, taking into account that
\begin{displaymath}
    \Big| \int_{t}^{\tau} \frac{z(\xi)}{(T - \xi)^{2 - \alpha}} \, \rd \xi \Big|
    \leq \int_{t}^{\tau} \frac{\bar{H} (\xi - t)}{(T - \xi)^{2 - \alpha}} \, \rd \xi
    \leq \frac{\bar{H} (\tau - t)^2}{(T - \tau)^{2 - \alpha}}
    = o(\tau - t)
\end{displaymath}
and
\begin{displaymath}
    \Big| \frac{z(\tau)}{(T - \tau)^{1 - \alpha}} - \frac{z(\tau)}{(T - t)^{1 - \alpha}}\Big|
    \leq \bar{H} (\tau - t) \Big( \frac{1}{(T - \tau)^{1 - \alpha}} - \frac{1}{(T - t)^{1 - \alpha}}\Big)
    = o(\tau - t),
\end{displaymath}
we derive
\begin{multline} \label{Val_ast_difference}
    \Val_\ast(\tau, x_\tau(\cdot)) - \Val_\ast(t, w(\cdot))
    = \frac{1}{\Gamma(\alpha)} \frac{z(\tau)}{(T - t)^{1 - \alpha}} + o(\tau - t) \\
    = \frac{1}{\Gamma(\alpha) (T - t)^{1 - \alpha}} \int_{t}^{\tau} (^C D^\alpha x)(\xi) \, \rd \xi + o(\tau - t).
\end{multline}
Here and below, by the same symbol $o(\tau - t)$, we denote different functions with the property that $o(\tau - t) \rightarrow 0$ when $\tau \downarrow t$.
Thus, the functional $\Val_\ast$ is $ci$-differentiable of the order $\alpha$ at $(t, w(\cdot))$, and
\begin{displaymath}
    \partial_t^\alpha \Val_\ast(t, w(\cdot))
    = 0,
    \quad \nabla^\alpha \Val_\ast(t, w(\cdot))
    = \frac{1}{\Gamma(\alpha) (T - t)^{1 - \alpha}}.
\end{displaymath}
Moreover, due to \cref{proposition_dist}, the functional $G^0 \ni (t, w(\cdot)) \rightarrow (T - t)^{\alpha - 1} \in \mathbb{R}$ is continuous, and, therefore, $\Val_\ast$ is $ci$-smooth of the order $\alpha$.

Now, let us consider the functional
\begin{equation} \label{example_varphi}
    \varphi(t, w(\cdot))
    = \begin{cases}
        \big(\Val_\ast(t, w(\cdot)) + (T - t)^\alpha\big)^2, & \mbox{if } \Val_\ast(t, w(\cdot)) < - (T - t)^\alpha, \\
        0, & \mbox{if } |\Val_\ast(t, w(\cdot))| \leq (T - t)^\alpha, \\
        \big(\Val_\ast(t, w(\cdot)) - (T - t)^\alpha\big)^2, & \mbox{if } \Val_\ast(t, w(\cdot)) > (T - t)^\alpha,
    \end{cases}
\end{equation}
where $(t, w(\cdot)) \in G$.
Let us show that $\varphi$ satisfies all the assumptions of \cref{theorem_varphi_rho}.
First of all, we note that continuity of $\varphi$ follows from continuity of $\Val_\ast$ and the functional $G \ni (t, w(\cdot)) \rightarrow (T - t)^\alpha \in \mathbb{R}$ (see \cref{proposition_dist}).
Further, let us show that $\varphi$ is $ci$-differentiable of the order $\alpha$ at every $(t, w(\cdot)) \in G^0$.
Let us fix $(t, w(\cdot)) \in G^0$ and $x(\cdot) \in X(t, w(\cdot))$ and consider the following cases.

Let us suppose that $\Val_\ast(t, w(\cdot)) < - (T - t)^\alpha$.
Then, due to continuity of the function $[0, T] \ni \tau \mapsto (\tau, x_\tau(\cdot)) \in G$ (see \cref{proposition_dist}), there exists $\delta \in (0,  T - t)$ such that $\Val_\ast(\tau, x_{\tau}(\cdot)) < - (T - \tau)^\alpha$ for $\tau \in (t, t + \delta]$.
Hence, for $\tau \in (t, t + \delta]$, we derive
\begin{displaymath}
    \varphi(\tau, x_\tau (\cdot)) - \varphi(t, w(\cdot))
    = \big(\Val_\ast(\tau, x_\tau(\cdot)) + (T - \tau)^\alpha\big)^2
    - \big(\Val_\ast(t, w(\cdot)) + (T - t)^\alpha \big)^2,
\end{displaymath}
and, consequently, taking \cref{Val_ast_difference} into account, we get
\begin{multline*}
    \varphi(\tau, x_\tau (\cdot)) - \varphi(t, w(\cdot))
    = 2 \big(\Val_\ast(t, w(\cdot)) + (T - t)^\alpha \big) \\
    \times \Big( - \frac{\alpha}{(T - t)^{1 - \alpha}} (\tau - t)
    + \frac{1}{\Gamma(\alpha) (T - t)^{1 - \alpha}} \int_{t}^{\tau} (^C D^\alpha x)(\xi) \, \rd \xi \Big)
    + o(\tau - t).
\end{multline*}
Hence, $\varphi$ is $ci$-differentiable of the order $\alpha$ at $(t, w(\cdot))$, and
\begin{displaymath}
    \partial_t^\alpha \varphi(t, w(\cdot))
    = - 2 \alpha \frac{\Val_\ast(t, w(\cdot)) + (T - t)^\alpha}{(T - t)^{1 - \alpha}},
    \quad \nabla^\alpha \varphi(t, w(\cdot))
    = 2 \frac{\Val_\ast(t, w(\cdot)) + (T - t)^\alpha}{\Gamma(\alpha) (T - t)^{1 - \alpha}}.
\end{displaymath}

Further, if $|\Val_\ast(t, w(\cdot))| < (T - t)^\alpha$, then it is clear that $\varphi$ is $ci$-differentiable of the order $\alpha$ at $(t, w(\cdot))$, and $\partial_t^\alpha \varphi(t, w(\cdot)) = \nabla^\alpha \varphi(t, w(\cdot)) = 0$.

Now, let $\Val_\ast(t, w(\cdot)) = - (T - t)^\alpha$.
Then, there exists $\delta \in (0,  T - t)$ such that, for every $\tau \in (t, t + \delta]$, only two cases are possible: either $|\Val_\ast(\tau, x_{\tau}(\cdot))| \leq (T - \tau)^\alpha$ or $\Val_\ast(\tau, x_{\tau}(\cdot)) < - (T - \tau)^\alpha$.
In the first case, we have $\varphi(\tau, x_\tau (\cdot)) - \varphi(t, w(\cdot)) = 0$.
In the second case, we derive
\begin{multline*}
    \varphi(\tau, x_\tau (\cdot)) - \varphi(t, w(\cdot))
    = \big(\Val_\ast(\tau, x_{\tau}(\cdot)) - \Val_\ast(t, w(\cdot)) + (T - \tau)^\alpha - (T - t)^\alpha \big)^2 \\
    = \Big(- \frac{\alpha}{(T - t)^{1 - \alpha}} (\tau - t)
    + \frac{1}{\Gamma(\alpha) (T - t)^{1 - \alpha}} \int_{t}^{\tau} (^C D^\alpha x)(\xi) \, \rd \xi
    + o(\tau - t) \Big)^2 \\
    = o(\tau - t).
\end{multline*}
So, $\varphi$ is $ci$-differentiable of the order $\alpha$ at $(t, w(\cdot))$, $\partial_t^\alpha \varphi(t, w(\cdot)) = \nabla^\alpha \varphi(t, w(\cdot)) = 0$.

By similar arguments, we obtain that $\varphi$ is $ci$-differentiable of the order $\alpha$ at $(t, w(\cdot))$ such that $\Val_\ast(t, w(\cdot)) \geq (T - t)^\alpha$, and
\begin{displaymath}
    \partial_t^\alpha \varphi(t, w(\cdot))
    = 2 \alpha \frac{\Val_\ast(t, w(\cdot)) - (T - t)^\alpha}{(T - t)^{1 - \alpha}},
    \quad \nabla^\alpha \varphi(t, w(\cdot))
    = 2 \frac{\Val_\ast(t, w(\cdot)) - (T - t)^\alpha}{\Gamma(\alpha) (T - t)^{1 - \alpha}}.
\end{displaymath}

Thus, the functional $\varphi$ is $ci$-differentiable of the order $\alpha$ at every $(t, w(\cdot)) \in G^0$.
Moreover, the analysis of the calculated above values of $\partial_t^\alpha \varphi(t, w(\cdot))$ and $\nabla^\alpha \varphi(t, w(\cdot))$ shows that the functionals $\partial_t^\alpha \varphi: G^0 \rightarrow \mathbb{R}$ and $\nabla^\alpha \varphi: G^0 \rightarrow \mathbb{R}$ are continuous, and, therefore, $\varphi$ is $ci$-smooth of the order $\alpha$.

In the considered optimal control problem, $\H(s) = - \Gamma(\alpha + 1) |s|$, $s \in \mathbb{R}$ (see \cref{Hamiltonian}).
Then, Hamilton--Jacobi--Bellman equation \cref{HJB} takes the following form:
\begin{displaymath}
    \partial_t^\alpha \varphi(t, w(\cdot))
    - \Gamma(\alpha + 1) |\nabla^\alpha \varphi(t, w(\cdot))|
    = 0,
    \quad (t, w(\cdot)) \in G^0.
\end{displaymath}
By direct substitution of the values of $\partial_t^\alpha \varphi(t, w(\cdot))$ and $\nabla^\alpha \varphi(t, w(\cdot))$, we conclude that $\varphi$ satisfies this equation.
Further, according to \cref{ID,example_Val_ast,example_varphi}, we have
\begin{displaymath}
    \varphi(T, w(\cdot))
    = \Val_\ast^2(T, w(\cdot))
    = w^2(T),
    \quad w(\cdot) \in \AC([0, T], \mathbb{R}),
\end{displaymath}
and, therefore, $\varphi$ satisfies right-end condition \cref{terminal_condition}.
Consequently, by \cref{theorem_varphi_rho}, $\varphi$ is the value functional of the considered optimal control problem, i.e., $\Val(t, w(\cdot)) = \varphi(t, w(\cdot))$, $(t, w(\cdot)) \in G$, and an optimal control strategy can be determined by
\begin{displaymath}
    U^\circ(t, w(\cdot))
    \in \begin{cases}
        \{ 1 \}, & \mbox{if } \Val_\ast(t, w(\cdot)) < - (T - t)^\alpha, \\
        [-1, 1], & \mbox{if } |\Val_\ast(t, w(\cdot))| \leq (T - t)^\alpha, \\
        \{ - 1\}, & \mbox{if } \Val_\ast(t, w(\cdot)) > (T - t)^\alpha,
    \end{cases}
\end{displaymath}
where $(t, w(\cdot)) \in G^0$.
The example is completely solved.

\begin{remark}
    In the case $\alpha = 1$, the obtained solution agrees with the classical one.
    Namely, since $\Val_\ast(t, w(\cdot)) = \Val_\ast(t, w(t)) = w(t)$, then
    \begin{displaymath}
        \Val(t, w(\cdot))
        = \Val(t, w(t))
        = \begin{cases}
            (w(t) + T - t)^2, & \mbox{if } w(t) < t - T, \\
            0, & \mbox{if } |w(t)| \leq T - t, \\
            (w(t) - T + t)^2, & \mbox{if } w(t) > T - t,
        \end{cases}
    \end{displaymath}
    where $(t, w(\cdot)) \in G$, and
    \begin{displaymath}
        U^\circ(t, w(\cdot))
        = U^\circ(t, w(t))
        \in \begin{cases}
            \{ 1 \}, & \mbox{if } w(t) < t - T, \\
            [-1, 1], & \mbox{if } |w(t)| \leq T - t, \\
            \{ - 1\}, & \mbox{if } w(t) > T - t,
        \end{cases}
    \end{displaymath}
    where $(t, w(\cdot)) \in G^0$.
\end{remark}

\section{Conclusions}
\label{section_conclusion}

In the paper, we have studied a Bolza-type optimal control problem for a fractional-order dynamical system.
We have proposed to consider the value of this problem as a functional in a suitable space of histories of motions.
We have proved that, within this approach, the dynamic programming principle is satisfied.
Further, we have introduced a new notion of fractional coinvariant differentiation of functionals and associated the optimal control problem with a Hamilton--Jacobi--Bellman equation with the fractional coinvariant derivatives.
Under certain smoothness assumptions, we have established a connection between the value functional and solutions to this equation.
In particular, we have proposed a way of constructing an optimal control strategy.
The obtained results have been illustrated by an example.

Nevertheless, it should be noted that the value functional of the considered optimal control problem may not possess the specified smoothness properties.
Therefore, future research will be devoted to developing the theory of generalized (minimax and viscosity) solutions to the obtained Hamilton--Jacobi--Bellman equation and studying their connection with the value functional in the general non-smooth case.

\bibliographystyle{siamplain}
\bibliography{references}

\end{document}